\documentclass[11pt]{article}

\usepackage{epsfig,epsf,fancybox}
\usepackage{amsmath}
\usepackage{mathrsfs}
\usepackage{amssymb}
\usepackage{graphicx}
\usepackage{color}
\usepackage[linktocpage,pagebackref,colorlinks,linkcolor=blue,anchorcolor=blue,citecolor=blue,urlcolor=blue,hypertexnames=false]{hyperref}
\usepackage{boxedminipage}
\usepackage{stmaryrd}
\usepackage{multirow}
\usepackage{booktabs}
\usepackage{accents}
\usepackage{cite}
\usepackage{float}
\usepackage[ruled,vlined,linesnumbered]{algorithm2e}

\newtheorem{theorem}{Theorem}[section]

\newtheorem{lemma}[theorem]{Lemma}
\newtheorem{proposition}[theorem]{Proposition}

\,
\,

\newcommand{\be}{\begin{equation}}
\newcommand{\ee}{\end{equation}}
\newcommand{\ba}{\begin{array}}
\newcommand{\ea}{\end{array}}
\newcommand{\bpm}{\begin{pmatrix}}
\newcommand{\epm}{\end{pmatrix}}

\newcommand{\bea}{\begin{eqnarray}}
\newcommand{\eea}{\end{eqnarray}}
\newcommand{\beaa}{\begin{eqnarray*}}
\newcommand{\eeaa}{\end{eqnarray*}}

\newcommand{\bal}{\begin{align}}
\newcommand{\eal}{\end{align}}
\newcommand{\baln}{\begin{align*}}
\newcommand{\ealn}{\end{align*}}

\usepackage[normalem]{ulem} 

\newcommand{\vu}{u}
\newcommand{\vv}{v}

\newcommand{\vW}{W}

\newcommand{\cD}{{\mathcal{D}}}

\newcommand{\cK}{{\mathcal{K}}}
\newcommand{\cL}{{\mathcal{L}}}

\newcommand{\cN}{{\mathcal{N}}}

\newcommand{\cX}{{\mathcal{X}}}
\newcommand{\cY}{{\mathcal{Y}}}

\newcommand{\RR}{\mathbb{R}} 

\newcommand{\st}{\mbox{ s.t. }}

\DeclareMathOperator*{\argmin}{arg\,min} 

\newcommand{\bc}{\begin{center}}
\newcommand{\ec}{\end{center}}

\newcommand{\bdm}{\begin{displaymath}}
\newcommand{\edm}{\end{displaymath}}

\newcommand{\beq}{\begin{equation}}
\newcommand{\eeq}{\end{equation}}

\newcommand{\bfl}{\begin{flushleft}}
\newcommand{\efl}{\end{flushleft}}

\newcommand{\bt}{\begin{tabbing}}
\newcommand{\et}{\end{tabbing}}

\newcommand{\beqn}{\begin{eqnarray}}
\newcommand{\eeqn}{\end{eqnarray}}

\newcommand{\beqs}{\begin{align*}} 
\newcommand{\eeqs}{\end{align*}}  

\newtheorem{remark}{Remark}[section]
\newtheorem{assumption}{Assumption}

\begin{document}

\title{Accelerated first-order primal-dual proximal methods for linearly constrained composite convex programming}

\author{Yangyang Xu\thanks{\url{yxu76@ua.edu}. Department of Mathematics, University of Alabama, Tuscaloosa, AL}
}

\date{}

\maketitle

\begin{abstract}
Motivated by big data applications, first-order methods have been extremely popular in recent years. However, naive gradient methods generally converge slowly. Hence, much efforts have been made to accelerate various first-order methods. This paper proposes two accelerated methods towards solving structured linearly constrained convex programming, for which we assume composite convex objective that is the sum of a differentiable function and a possibly nondifferentiable one.

The first method is the accelerated linearized augmented Lagrangian method (LALM). At each update to the primal variable, it allows linearization to the differentiable function and also the augmented term, and thus it enables easy subproblems. Assuming merely weak convexity, we show that LALM owns $O(1/t)$ convergence if parameters are kept fixed during all the iterations and can be accelerated to  $O(1/t^2)$ if the parameters are adapted, where $t$ is the number of total iterations.

The second method is the accelerated linearized alternating direction method of multipliers (LADMM). In addition to the composite convexity, it further assumes two-block structure on the objective. Different from classic ADMM, our method allows linearization to the objective and also augmented term to make the update simple. Assuming strong convexity on one block variable, we show that LADMM also enjoys $O(1/t^2)$ convergence with adaptive parameters.  This result is a significant improvement over that in [Goldstein et. al, SIIMS'14], which requires strong convexity on both block variables and no linearization to the objective or augmented term. 

Numerical experiments are performed on quadratic programming, image denoising, and support vector machine. The proposed accelerated methods are compared to nonaccelerated ones and also existing accelerated methods. The results demonstrate the validness of acceleration and superior performance of the proposed methods over existing ones.

\vspace{0.3cm}

\vspace{0.3cm}


\end{abstract}

\section{Introduction}
In recent years, motivated by applications that involve extremely big data, first-order methods with or without splitting techniques have received tremendous attention in a variety of areas such as statistics, machine learning, data mining, and image processing. Compared to traditional methods like the Newton's method, first-order methods only require gradient information instead of the much more expensive Hessian. Splitting techniques can further decompose a single difficult large-scale problem into smaller and easier ones. However, in both theory and practice, first-order methods often converge slowly if no additional techniques are applied. For this reason, lots of efforts have been made to accelerate various first-order methods.

In this paper, we consider the linearly constrained problem
\begin{equation}\label{eq:lc-prob}
\min_x F(x), \st Ax=b,
\end{equation}
where $F$ is a proper closed convex but possibly nondifferentiable function. We allow $F$ to be extended-valued, and thus in addition to the linear constraint, \eqref{eq:lc-prob} can also include the constraint $x\in\cX$ if part of $F$ is the indicator function of a convex set $\cX$.

The augmented Lagrangian method (ALM) \cite{bertsekas2014constrained}  is one most popular approach to solve constrained optimization problems like \eqref{eq:lc-prob}. Let 
\begin{equation}\label{eq:aug-fun}\cL_\beta(x,\lambda)=F(x)-\langle\lambda, Ax-b\rangle+\frac{\beta}{2}\|Ax-b\|^2
\end{equation}
be the augmented Lagrangian function. Then ALM for \eqref{eq:lc-prob} iteratively performs the updates
\begin{subequations}\label{eq:alm}
\begin{align}
&x^{k+1}\in\argmin_x \cL_\beta(x, \lambda^k),\label{eq:alm-x}\\
&\lambda^{k+1}=\lambda^k-\beta(Ax^{k+1}-b).\label{eq:alm-lam}
\end{align}
\end{subequations}
In general, the subproblem \eqref{eq:alm-x} may not have a solution or have more than one solutions, and even if a unique solution exists, it could be difficult to find the solution. We will assume certain structures of $F$ and also modify the updates in \eqref{eq:alm} to have well-defined and easier subproblems.

\subsection{Linearized ALM for linearly constrained composite convex problems}
We first assume the composite convexity structure, i.e., the objective in \eqref{eq:lc-prob} can be written as:  \begin{equation}\label{eq:obj-lalm}F(x)=f(x)+g(x),\end{equation} 
where $f$ is a convex Lipschitz differentiable function, and $g$ is a proper closed convex but possibly nondifferentiable function. 
Hence, the problem \eqref{eq:lc-prob} reduces to the linearly constrained composite convex programming:
\begin{equation}\label{eq:prob-lalm}
\min_x f(x)+g(x), \st Ax=b.
\end{equation}
Usually, $g$ is simple such as the indicator function of the nonnegative orthant or $\ell_1$-norm, but the smooth term $f$ could be complicated like the logistic loss function. 

Our first modification to the update in \eqref{eq:alm-x} is to approximate $f$ by a simple funtion. Typically, we replace $f$ by a quadratic function that dominates $f$ around $x^k$, resulting in the linearized ALM as follows: 
\begin{subequations}\label{eq:lalm}
\begin{align}
&x^{k+1}\in\argmin_x \langle \nabla f(x^k)-A^\top\lambda^k, x\rangle+g(x)+\frac{\beta}{2}\|Ax-b\|^2+\frac{1}{2}\|x-x^k\|_P^2,\label{eq:lalm-x}\\
&\lambda^{k+1}=\lambda^k-\beta(Ax^{k+1}-b),\label{eq:lalm-lam}
\end{align}
\end{subequations}
where the weight matrix $P$ is positive semidefinite (PSD) and can be set according to the Lipschitz constant of $\nabla f$. Choosing appropriate $P$ like $\eta I-\beta A^\top A$, we can also linearize the augmented term and have a closed form solution if $g$ is simple.

The linearization technique here is not new. It is commonly used in the proximal gradient method, which can be regarded as a special case of \eqref{eq:lalm} by removing the linear constraint $Ax=b$. It has also been used in the linearized alternating direction method of multipliers (ADMM) \cite{ouyang2015accelerated} and certain primal-dual methods (e.g., \cite{condat2013primal, gao2015first, GXZ-RPDCU2016}).

Our second modification is to adaptively choose the parameters in the linearized ALM and also linearize $f$ at a point other than $x^k$ to accelerate the convergence of the method. Algorithm \ref{alg:alalm} summarizes the proposed accelerated linearized ALM. The idea of using three point sequences for acceleration is first adopted in \cite{lan2012optimal}, and recently it is used in \cite{ouyang2015accelerated} to accelerate the linearized ADMM. 
\begin{algorithm}\caption{Accelerated linearized augmented Lagrangian method for \eqref{eq:prob-lalm}}\label{alg:alalm}
\DontPrintSemicolon
\textbf{Initialization:} choose $\bar{x}^1=x^1$ and set $\lambda^1=0.$\;
\For{$k=1,2,\ldots$}{
Choose parameters $\alpha_k,\beta_k,\gamma_k$ and $P^k$ and perform updates:
\begin{align}
&\hat{x}^k=(1-\alpha_k)\bar{x}^k+\alpha_k x^k,\label{eq:new-xhat}\\
&x^{k+1}\in\argmin_x \langle \nabla f(\hat{x}^k)-A^\top \lambda^k, x\rangle +g(x)+\frac{\beta_k}{2}\|Ax-b\|^2+\frac{1}{2}\|x-x^k\|_{P^k}^2,\label{eq:new-x}\\
&\bar{x}^{k+1}=(1-\alpha_k)\bar{x}^k+\alpha_kx^{k+1},\label{eq:new-xbar}\\
&\lambda^{k+1}=\lambda^k-\gamma_k(Ax^{k+1}-b).\label{eq:new-lambda}
\end{align}
\If{A stopping condition is satisfied}{
Return $(x^{k+1},\bar{x}^{k+1},\lambda^{k+1})$.
}
}
\end{algorithm}

\subsection{Linearized ADMM for two-block structured problems}
In this section, we explore more structures of $F$. In addition to the composite convexity structure, we assume that the variable $x$ and accordingly the matrix $A$ can be partitioned into two blocks, i.e., 
\begin{equation}\label{eq:chg-var}
x=(y,z), \quad A=(B,C),
\end{equation} and the objective can be written as
\begin{equation}\label{eq:F-padmm}
F(x)=f(y)+g(z)+h(z),
\end{equation}
where $f$ and $g$ are proper closed convex but possibly nondifferentiable functions, and $h$ is a convex Lipschitz differentiable function. Hence, the problem \eqref{eq:lc-prob} reduces to the linearly constrained two-block structured problem:
\begin{equation}\label{eq:prob-padmm}
\min_{y,z} f(y)+g(z)+h(z), \st By+Cz=b.
\end{equation}

ADMM \cite{Glowinski1975, gabay1976dual} is a popular method that explores the two-block structure of \eqref{eq:prob-padmm} by alternatingly updating $y$ and $z$, followed by an update to the multiplier $\lambda$. More precisely, it iteratively performs the updates:
\begin{subequations}\label{eq:admm}
\begin{align}
&y^{k+1}\in\argmin_y \cL_\beta(y, z^k, \lambda^k),\label{eq:admm-y}\\
&z^{k+1}\in\argmin_z \cL_\beta(y^{k+1}, z, \lambda^k),\label{eq:admm-z}\\
&\lambda^{k+1}=\lambda^k-\beta(By^{k+1}+Cz^{k+1}-b),\label{eq:admm-lam}
\end{align}
\end{subequations}
where $\cL_\beta$ is given in \eqref{eq:aug-fun} with the notation in \eqref{eq:chg-var} and \eqref{eq:F-padmm}.
  It can be regarded as an inexact ALM, in the sense that it only finds an approximate solution to \eqref{eq:alm-x}. If \eqref{eq:admm-y} and \eqref{eq:admm-z} are run repeatedly before updating $\lambda$, a solution to \eqref{eq:alm-x} would be found, and thus the above update scheme reduces to that in \eqref{eq:alm}. However, one single run of \eqref{eq:admm-y} and \eqref{eq:admm-z}, followed by an update to $\lambda$, is sufficient to guarantee the convergence. Thus ADMM is often perferable over ALM on solving the two-block structured problem \eqref{eq:prob-padmm} since updating $y$ and $z$ separately could be much cheaper than updating them jointly.
 
Usually $f$ and $g$ are simple, but the smooth term $h$ in \eqref{eq:prob-padmm} could be complicated and thus make the $z$-update in \eqref{eq:admm-z} difficult. We apply the same linearization technique as in \eqref{eq:lalm-x} to \eqref{eq:admm-z} and in addition adaptively choose the parameters to accelerate the method. Algorithm \ref{alg:apadmm} summarizes the accelerated linearized ADMM. If $f$ and $g$ are simple, we can have closed form solutions to \eqref{eq:apadmm-y} and \eqref{eq:apadmm-z} by choosing appropriate $P^k$ and $Q^k$ to linearize the augmented terms.
\begin{algorithm}\caption{Accelerated linearized alternating direction method of multipliers for \eqref{eq:prob-padmm}}\label{alg:apadmm}
\DontPrintSemicolon
\textbf{Initialization:} choose $(y^1, z^1)$ and set $\lambda^1=0$.\;
\For{$k=0,1,2,\ldots$}{
Choose parameters $\beta_k,\gamma_k, P^k$ and $Q^k$ and perform updates:
\begin{subequations}\label{eq:apadmm}
\begin{align}
&y^{k+1}=\argmin_y f(y)-\langle \lambda^k, By\rangle + \frac{\beta_k}{2}\|By+Cz^k-b\|^2+\frac{1}{2}\|y-y^k\|_{P^k}^2,\label{eq:apadmm-y}\\
&z^{k+1}=\argmin_z g(z)+\langle \nabla h(z^k)-C^\top \lambda^k, z\rangle +\frac{\beta_k}{2}\|By^{k+1}+Cz-b\|^2+\frac{1}{2}\|z-z^k\|_{Q^k}^2,\label{eq:apadmm-z}\\
&\lambda^{k+1}=\lambda^k-\gamma_k(By^{k+1}+Cz^{k+1}-b).\label{eq:apadmm-lam} 
\end{align}
\end{subequations}
\If{A stopping condition is satisfied}{
Return $(y^{k+1},z^{k+1},\lambda^{k+1})$.
}
}
\end{algorithm}

\subsection{Related works}
It appears that \cite{nesterov1983method} is the first accelerated gradient method for general smooth convex programming. However, according to the google citation, the work does not really attract much attention until late 2010's. One possible reason could be that the problems people encountered before were not too large so second-order methods can handle them very efficiently. Since 2009, accelerated gradient methods have become extremely popular partly due to \cite{FISTA2009, nesterov2013gradient} that generalize the acceleration idea of \cite{nesterov1983method} to composite convex optimization problems and also due to the increasingly large scale problems arising in many areas. Both \cite{FISTA2009, nesterov2013gradient} achieve optimal rate for first-order methods, but their acceleration techniques look quite different. The former is essentially based on an extrapolation technique while the latter relies on a sequence of estimate functions with adaptive parameters. The recent work \cite{wibisono2016variational} studies a few accelerated methods from a continuous-time perspective. It is unclear how to apply that idea to primal-dual methods.

Although the methods in \cite{FISTA2009, nesterov2013gradient} can conceptually handle constrained problems, they require simple projection to the constraint set. Hence, they are not really good choices if we consider the structured linearly constrained problem \eqref{eq:prob-lalm} or \eqref{eq:prob-padmm}. However, the acceleration idea can still be applied. The ALM method in \eqref{eq:alm} is accelerated in \cite{he2010aalm} by using an extrapolation technique similar to that in \cite{FISTA2009} to the multiplier $\lambda$. While \cite{he2010aalm} requires the objective to be smooth, \cite{kang2013accelerated} extends it to general convex problems, and \cite{kang2015inexact} further reduces the requirement of exactly solving subproblems by assuming strong convexity of the objective. All these accelerated ALM methods do not consider any linearization to the objective or the augmented term. One exception is \cite{huang2013acLBM} that linearizes the augmented term and requires strong convexity of the primal problem in its analysis. Therefore, towards finding a solution to \eqref{eq:prob-lalm}, they may need to solve difficult subproblems if the smooth term $f$ is complicated.

The extrapolation technique in \cite{FISTA2009} has also been applied to accelerate the ADMM method in \cite{goldstein2014fast} for solving two-block structured problems like \eqref{eq:prob-padmm}. It requires both $f$ and $g+h$ to be strongly convex, and the extrapolation is performed to the multiplier and the secondly updated block variable. In addition, \cite{goldstein2014fast} does not consider linearization to the smooth term $h$ or the augmented term, and hence its applicability is restricted.  Although the acceleration is observed empirically in \cite{goldstein2014fast} for weakly convex problems, no convergence rate has been shown. A later work \cite{kadkhodaie2015accelerated} accelerates the nonlinearized ADMM by renewing the second updated block variable again after extrapolating the multiplier. It still requires strong convexity on both $f$ and $g+h$. Without assuming any strong convexity to the objective function, \cite{ouyang2015accelerated} achieves partial acceleration on linearized ADMM for solving problems in the form of \eqref{eq:prob-padmm}. It shows that the decaying rate related to the gradient Lipschitz constant $L_h$ can be $O(1/t^2)$ while the rate for other parts remains $O(1/t)$, where $t$ is the number of iterations. Without the linear constraint, the result in \cite{ouyang2015accelerated} matches the optimal rate of first-order methods. 

Different from the extrapolation technique used in the above mentioned accelerated ALM and ADMM methods, \cite{ouyang2015accelerated} follows the work \cite{lan2012optimal} and uses three point sequences and adaptive parameters. Algorithm \ref{alg:alalm} employs the same idea, and our result indicates that the acceleration to the linearized ALM method is not only applied to the gradient Lipschitz constant but also  to other parts, i.e., full acceleration. To gain full acceleration to Algorithm \ref{alg:apadmm}, we will require either $f$ or $g+h$ to be strongly convex, which is strictly weaker than that assumed in \cite{goldstein2014fast}. This assumption is also made in several accelerated primal-dual methods for solving bilinear saddle-point problems, e.g., \cite{chambolle2011first, chen2014optimal, he2016accelerated, bredies2016accelerated}. The outstanding work \cite{chambolle2011first} presents a framework of primal-dual method for the problem:
\begin{equation}\label{eq:sd-prob}
\min_{x\in\cX}\max_{y\in\cY} \langle Kx, y\rangle +G(x)-F(y),
\end{equation}
where $G$ and $F$ are both proper closed convex functions, and $K$ is a bounded linear operator. It is shown in \cite{chambolle2011first} that the method has $O(1/t^2)$ convergence if either $F$ or $G$ is strongly convex. As shown in \cite{GXZ-RPDCU2016}, the primal-dual method presented in \cite{chambolle2011first} is a special case of linearized ADMM applied to the dual problem of \eqref{eq:sd-prob} about $y$. Hence, it can fall into one case of Algorithm \ref{alg:apadmm}. However, \cite{chambolle2011first} sets parameters in a different way from what we use to accelerate the more general linearized ADMM method; see the example in section \ref{sec:denois}. On solving \eqref{eq:sd-prob}, the Douglas-Rachford splitting method has recently been applied and also accelerated in \cite{bredies2016accelerated} by assuming one of $F$ and $G$ to be strongly convex. In addition, \cite{dang2014randomized} generalizes the work \cite{chambolle2011first} to multi-block structured problems, and the generalized method still enjoy $O(1/t^2)$ convergence if strong convexity is assumed. Without assuming strong convexity, \cite{chen2014optimal} proposes a new primal-dual method for the saddle-point problem \eqref{eq:sd-prob} and achieves partial acceleration similar to what achieved in \cite{ouyang2015accelerated}.

Acceleration techniques have also been applied to other types of methods to different problems such as in coordinate descent methods (e.g., \cite{xu2013block, lin2014accelerated, fercoq2015accelerated}) and stochastic approximation methods (e.g., \cite{lan2012optimal, ghadimi2016accelerated}). Extending our discussion to these methods will be out of the scope of this paper. Interested readers are referred to those papers we mention here and the references therein.

\subsection{Contributions} We summarize our main contributions below.
\begin{itemize}
\item We propose an accelerated linearized ALM method for solving linearly constrained composite convex programming. By linearizing the possibly complicated smooth term in the objective, the method enables easy subproblems. Our acceleration strategy follows \cite{ouyang2015accelerated} that considers accelerated linearized ADMM method. Different from partial acceleration achieved in \cite{ouyang2015accelerated}, we obtain full acceleration and achieve the optimal $O(1/t^2)$ convergence rate by assuming merely weak convexity. 
\item We also propose an accelerated linearized ADMM method for solving two-block structured linearly constrained convex programming, where in the objective, one block variable has composite convexity structure. While \cite{goldstein2014fast} requires strong convexity on both block variables to achieve $O(1/t^2)$ convergence for nonlinearized ADMM, we only need strong convexity on one of them. Furthermore, linearization is allowed to the smooth term in the objective and also to the augmented Lagrangian term, and thus the method enables much easier subproblems than those for nonlinearized ADMM.
\item We test the proposed methods on quadratic programming, total variation regularized image denoising problem, and the elastic net regularized support vector machine. We compare them to nonaccelerated methods and also two other accelerated first-order methods. The numerical results demonstrate the validness of acceleration and also superiority of the proposed accelerated methods over other accelerated ones.
\end{itemize}

\subsection{Outline} The rest of the paper is organized as follows. In section \ref{sec:analysis}, we analyze Algorithm \ref{alg:alalm} and Algorithm \ref{alg:apadmm} with both fixed and adaptive parameters. Numerical experiments are performed in section \ref{sec:numerical}, and finally section \ref{sec:conclusion} concludes the paper and presents some interesting open questions. 

\section{Convergence analysis}\label{sec:analysis}
In this section, we analyze the convergence of Algorithms \ref{alg:alalm} and \ref{alg:apadmm}. Assuming merely weak convexity, we show that Algorithm \ref{alg:alalm} with adaptive parameters enjoys a fast convergence with rate $O(1/t^2)$, where $t$ is the number of total iterations. For Algorithm \ref{alg:apadmm}, we establish the same order of convergence rate by assuming strong convexity on the $z$-part.

\subsection{Notation and preliminary lemmas} 
Before proceeding with our analysis, let us introduce some notation and preliminary lemmas.

We denote $\cX^*$ as the solution set of \eqref{eq:lc-prob}. A point $x^*$ is a solution to \eqref{eq:lc-prob} if there exists $\lambda^*$ such that the KKT conditions hold:
\begin{subequations}\label{eq:kkt-lalm}
\begin{align}
&0\in\partial F(x^*)-A^\top\lambda^*,\label{eq:kkt-lalm-d}\\
&Ax^*-b=0,\label{eq:kkt-lalm-p}
\end{align}
\end{subequations}
Together with the convexity of $F$, the conditions in \eqref{eq:kkt-lalm} implies that
\begin{equation}\label{eq:opt-lalm}
F(x)-F(x^*)-\langle\lambda^*, Ax-b\rangle\ge 0,\,\forall x.
\end{equation}
For any vector $v$ and any symmetric matrix $W$ of appropriate size, we define $\|v\|_W^2=v^\top Wv$. Note this definition does not require positive semidefiniteness of $W$.

\begin{lemma}
For any two vectors $\vu, \vv$ and a symmetric matrix $\vW$, we have
\begin{equation}\label{uv-cross}
2\vu^\top \vW \vv = \|\vu\|_\vW^2+\|\vv\|_\vW^2-\|\vu-\vv\|_\vW^2.
\end{equation}
\end{lemma}

\begin{lemma}\label{lem:x-rate}
Given a function $\phi$ and a fixed point $\tilde{x}$, if for any $\lambda$, it holds that
\begin{equation}\label{eq:x-F}F(\tilde{x})-F(x^*)-\langle\lambda, A\tilde{x}-b\rangle\le \phi(\lambda),
\end{equation}
then for any $\rho>0$, we have
\begin{equation}\label{eq:sup-F} 
F(\tilde{x})-F(x^*)+\rho \|A\tilde{x}-b\|\le \sup_{\|\lambda\|\le \rho}\phi(\lambda).
\end{equation}
\end{lemma}
This lemma can be found in \cite{GXZ-RPDCU2016}. Here we provide a simple proof.
\begin{proof}
If $A\tilde{x}=b$, then it is trivial to have \eqref{eq:sup-F} from \eqref{eq:x-F}. Otherwise, let $\lambda=-\frac{\rho(A\tilde{x}-b)}{\|A\tilde{x}-b\|}$ in both sides of \eqref{eq:x-F} and the result follows by noting 
$$\phi\left(-\frac{\rho(A\tilde{x}-b)}{\|A\tilde{x}-b\|}\right)\le \sup_{\|\lambda\|\le \rho}\phi(\lambda).$$
\end{proof}

\begin{lemma} \label{lem:equiv-rate}
For any $\epsilon\ge0$, if
\begin{equation}\label{eq:ineq-eps}
F(\tilde{x})-F(x^*)+\rho\|A\tilde{x}-b\|\le \epsilon,
\end{equation}
then we have
\begin{equation}\label{eq:ineq-rate}
\|A\tilde{x}-b\| \leq \frac{\epsilon}{\rho-\|\lambda^*\|} \mbox{ and }  -\frac{\|\lambda^*\|\epsilon}{\rho-\|\lambda^*\|}\le F(\tilde{x})-F(x^*) \leq \epsilon,
\end{equation}
where $(x^*,\lambda^*)$ satisfies the KKT conditions in \eqref{eq:kkt-lalm}, and we assume $\|\lambda^*\|< \rho$.
\end{lemma}
\begin{proof}
From \eqref{eq:opt-lalm}, we have
$$F(\tilde{x})-F(x^*)\ge -\|\lambda^*\|\cdot\|A\tilde{x}-b\|,$$
which together with \eqref{eq:ineq-eps} implies the first inequality in \eqref{eq:ineq-rate}. The other two inequalities follow immediately.
\end{proof}

\subsection{Analysis of the accelerated linearized ALM}
In this subsection, we show the convergence of Algorithm \ref{alg:alalm} under the following assumptions.
\begin{assumption}\label{lalm-assump1}
There exists a point $(x^*,\lambda^*)$ satisfying the KKT conditions in \eqref{eq:kkt-lalm}.
\end{assumption}

\begin{assumption}\label{lalm-assump2}
The function $f$ has Lipschitz continuous gradient with constant $L_f$, i.e.,
\begin{equation}\label{eq:lip-gradf}\|\nabla f(x)-\nabla f(\tilde{x})\|\le L_f\|x-\tilde{x}\|,\,\forall x, \tilde{x}.
\end{equation}
\end{assumption}
The inequality in \eqref{eq:lip-gradf} implies that
\begin{equation}\label{eq:lip-ineq}
f(\tilde{x})\le f(x)+\langle\nabla f(x), \tilde{x}-x\rangle+\frac{L_f}{2}\|\tilde{x}-x\|^2,\,\forall x, \tilde{x}.
\end{equation}

We first establish a result of running one iteration of Algorithm \ref{alg:alalm}. The proof follows that in \cite{ouyang2015accelerated}.
\begin{lemma}[One-iteration result]\label{thm:lalm-1step}
Let $\{(x^k,\bar{x}^k,\lambda^k)\}_{k\ge1}$ be the sequence generated from Algorithm \ref{alg:alalm} with $0\le\alpha_k\le 1,\,\forall k$. Then for any $(x,\lambda)$ such that $Ax=b$, we have
\begin{align}\label{eq:lalm-1step}
&\big[F(\bar{x}^{k+1})-F(x)-\langle\lambda, A\bar{x}^{k+1}-b\rangle\big]-(1-\alpha_k)\big[F(\bar{x}^k)-F(x)-\langle\lambda, A\bar{x}^k-b\rangle\big]\cr
\le & -\frac{\alpha_k}{2}\big[\|x^{k+1}-x\|_{P^k}^2-\|x^k-x\|_{P^k}^2+\|x^{k+1}-x^k\|_{P^k}^2\big]+\frac{\alpha_k^2 L_f}{2}\|x^{k+1}-x^k\|^2\cr
&  +\frac{\alpha_k}{2\gamma_k}\big[\|\lambda^k-\lambda\|^2-\|\lambda^{k+1}-\lambda\|^2+\|\lambda^{k+1}-\lambda^k\|^2\big]-\frac{\alpha_k\beta_k}{\gamma_k^2}\|\lambda^{k+1}-\lambda^k\|^2,
\end{align}
where $F$ is given in \eqref{eq:obj-lalm}.
\end{lemma}

\begin{proof}
From \eqref{eq:lip-ineq}, it follows that
$$f(\bar{x}^{k+1})\le  f(\hat{x}^k)+\langle \nabla f(\hat{x}^k), \bar{x}^{k+1}-\hat{x}^k\rangle +\frac{L_f}{2}\|\bar{x}^{k+1}-\hat{x}^k\|^2.$$
Substituting $\bar{x}^{k+1}=(1-\alpha_k)\bar{x}^k+\alpha_k x^{k+1}$ and also noting $\bar{x}^{k+1}-\hat{x}^k=\alpha_k(x^{k+1}-x^k)$, we have from the above inequality that
\begin{align}\label{eq:bd1}
f(\bar{x}^{k+1})\le & f(\hat{x}^k)+(1-\alpha_k)\langle \nabla f(\hat{x}^k), \bar{x}^k-\hat{x}^k\rangle +\alpha_k\langle \nabla f(\hat{x}^k), x^{k+1}-\hat{x}^k\rangle+\frac{\alpha_k^2 L_f}{2}\|x^{k+1}-x^k\|^2\cr
=&(1-\alpha_k)\big[f(\hat{x}^k)+\langle \nabla f(\hat{x}^k), \bar{x}^k-\hat{x}^k\rangle\big]+\alpha_k\big[f(\hat{x}^k)+\langle \nabla f(\hat{x}^k), x-\hat{x}^k\rangle\big]\cr
&+\alpha_k\langle \nabla f(\hat{x}^k), x^{k+1}-x\rangle+\frac{\alpha_k^2 L_f}{2}\|x^{k+1}-x^k\|^2\cr
\le &(1-\alpha_k) f(\bar{x}^k)+\alpha_k f(x) +\alpha_k\langle \nabla f(\hat{x}^k), x^{k+1}-x\rangle+\frac{\alpha_k^2 L_f}{2}\|x^{k+1}-x^k\|^2,
\end{align}
where the second inequality follows from the convexity of $f$.
Hence,
\begin{align}\label{eq:bd2}
&\big[F(\bar{x}^{k+1})-F(x)-\langle\lambda, A\bar{x}^{k+1}-b\rangle\big]-(1-\alpha_k)\big[F(\bar{x}^k)-F(x)-\langle\lambda, A\bar{x}^k-b\rangle\big]\cr
=&\big[f(\bar{x}^{k+1})-(1-\alpha_k)f(\bar{x}^k)-\alpha_k f(x)\big]+\big[g(\bar{x}^{k+1})-(1-\alpha_k)g(\bar{x}^k)-\alpha_k g(x)\big]-\alpha_k\langle\lambda, Ax^{k+1}-b\rangle\cr
\le & \alpha_k\langle\nabla f(\hat{x}^k), x^{k+1}-x\rangle+\frac{\alpha_k^2 L_f}{2}\|x^{k+1}-x^k\|^2+\alpha_k[g(x^{k+1})-g(x)]-\alpha_k\langle\lambda, Ax^{k+1}-b\rangle,\end{align}
where the equality follows from the fact $\bar{x}^{k+1}=(1-\alpha_k)\bar{x}^k+\alpha_k x^{k+1}$, and in the inequality, we have used \eqref{eq:bd1} and the convexity of $g$.

On the other hand, from the update rule of $x^{k+1}$, we have the optimality condition:
$$0= \nabla f(\hat{x}^k)+\tilde{\nabla} g(x^{k+1})-A^\top \lambda^k +\beta_k A^\top(Ax^{k+1}-b)+P^k(x^{k+1}-x^k),$$
where $\tilde{\nabla} g(x^{k+1})$ is a subgradient of $g$ at $x^{k+1}$. 
Hence, for any $x$ such that $Ax=b$, it holds
\begin{align}\label{eq:bd3}
0=&\big\langle x^{k+1}-x, \nabla f(\hat{x}^k)+\tilde{\nabla} g(x^{k+1})-A^\top \lambda^k +\beta_k A^\top(Ax^{k+1}-b)+P^k(x^{k+1}-x^k)\big\rangle\cr
\ge & \big\langle x^{k+1}-x, \nabla f(\hat{x}^k)-A^\top \lambda^k +\beta_k A^\top(Ax^{k+1}-b)+P^k(x^{k+1}-x^k)\big\rangle +g(x^{k+1})-g(x)\cr
= &\left\langle x^{k+1}-x, \nabla f(\hat{x}^k)-A^\top \lambda^k+\frac{\beta_k}{\gamma_k}A^\top(\lambda^k-\lambda^{k+1})+P^k(x^{k+1}-x^k)\right\rangle +g(x^{k+1})-g(x)\cr
= &\big\langle x^{k+1}-x, \nabla f(\hat{x}^k)\big\rangle +g(x^{k+1})-g(x)+\big\langle x^{k+1}-x,P^k(x^{k+1}-x^k)\big\rangle\cr
& +\left\langle A(x^{k+1}-x), - \lambda^k+\frac{\beta_k}{\gamma_k}(\lambda^k-\lambda^{k+1})\right\rangle\cr
=&\big\langle x^{k+1}-x, \nabla f(\hat{x}^k)\big\rangle +g(x^{k+1})-g(x)+\big\langle x^{k+1}-x,P^k(x^{k+1}-x^k)\big\rangle\cr
& +\left\langle Ax^{k+1}-b, \lambda- \lambda^k+\frac{\beta_k}{\gamma_k}(\lambda^k-\lambda^{k+1})\right\rangle-\langle\lambda, Ax^{k+1}-b\rangle\cr
=&\big\langle x^{k+1}-x, \nabla f(\hat{x}^k)\big\rangle +g(x^{k+1})-g(x)-\langle\lambda, Ax^{k+1}-b\rangle+\big\langle x^{k+1}-x,P^k(x^{k+1}-x^k)\big\rangle\cr
& +\left\langle \frac{1}{\gamma_k}(\lambda^k-\lambda^{k+1}), \lambda- \lambda^k+\frac{\beta_k}{\gamma_k}(\lambda^k-\lambda^{k+1})\right\rangle
\end{align}
where the inequality follows from the convexity of $g$. 

Combining \eqref{eq:bd2} and \eqref{eq:bd3} together gives
\begin{align*}
&\big[F(\bar{x}^{k+1})-F(x)-\langle\lambda, A\bar{x}^{k+1}-b\rangle\big]-(1-\alpha_k)\big[F(\bar{x}^k)-F(x)-\langle\lambda, A\bar{x}^k-b\rangle\big]\cr
\le & \frac{\alpha_k^2 L_f}{2}\|x^{k+1}-x^k\|^2-\alpha_k\big\langle x^{k+1}-x,P^k(x^{k+1}-x^k)\big\rangle\\
&-\alpha_k\left\langle \frac{1}{\gamma_k}(\lambda^k-\lambda^{k+1}), \lambda- \lambda^k+\frac{\beta_k}{\gamma_k}(\lambda^k-\lambda^{k+1})\right\rangle.
\end{align*}
Now apply \eqref{uv-cross} to complete the proof.
\end{proof}

Below, we specify the values of the parameters $\alpha_k, \beta_k, \gamma_k$ and $P^k$ and establish the convergence rate of Algorithm \ref{alg:alalm} through \eqref{eq:lalm-1step}.

\subsubsection{Constant parameters}
In this subsection, we fix the parameters $\alpha_k, \beta_k, \gamma_k$ and $P^k$ during all the iterations and show $O(1/t)$ convergence of Algorithm \ref{alg:alalm}. The result is summarized in the following theorem. Note that this result is not totally new. Similar result is indicated by several previous works; see \cite{gao2015first, GXZ-RPDCU2016} for example. However, this special case seems to be overlooked in the literature. In addition, we notice that our result allows more flexible relation between $\beta$ and $\gamma$. Previous works usually assume $\beta=\gamma$ because they consider problems with at least two block variables.
\begin{theorem}\label{thm:alalm-1-t}
Under Assumptions \ref{lalm-assump1} and \ref{lalm-assump2}, let $\{(x^k,\bar{x}^k,\lambda^k)\}_{k\ge1}$ be the sequence generated from Algorithm \ref{alg:alalm} with parameters set to
\begin{equation}\label{eq:alalm-const-para}
\forall k:\, \alpha_k=1,\, \beta_k=\beta>0,\, \gamma_k=\gamma\in(0,2\beta),\,P_k=P\succ L_f I.
\end{equation}
Then $\bar{x}^k=x^k,\,\forall k$, and $\{(x^k,\lambda^k)\}_{k\ge 1}$ is bounded and converges to a point $(x^\infty,\lambda^\infty)$ that satisfies the KKT conditions in \eqref{eq:kkt-lalm}.
In addition, 
\begin{subequations}\label{eq:alalm-const-rate}
\begin{align}
&|F(\tilde{x}^{t+1})-F(x^*)|\le  \frac{1}{t}\left(\frac{1}{2}\|x^1-x^*\|_P^2+\frac{2\|\lambda^*\|^2}{\gamma}\right),\\
&\|A\tilde{x}^{t+1}-b\|\le  \frac{1}{t\|\lambda^*\|}\left(\frac{1}{2}\|x^1-x^*\|_P^2+\frac{2\|\lambda^*\|^2}{\gamma}\right),
\end{align}
\end{subequations}
where $(x^*,\lambda^*)$ is any point satisfying the KKT conditions in \eqref{eq:kkt-lalm}, and
$$\tilde{x}^{t+1}=\frac{\sum_{k=1}^t x^{k+1}}{t}.$$
\end{theorem}
\begin{proof}
It is trivial to have $\bar{x}^k=\hat{x}^k=x^k$ from \eqref{eq:new-xhat} and \eqref{eq:new-xbar} as $\alpha_k=1,\,\forall k$. With the parameters given in \eqref{eq:alalm-const-para} and $x=x^*$, the inequality in \eqref{eq:lalm-1step} reduces to
\begin{align}\label{alalm-const-bd1}
&F(x^{k+1})-F(x^*)-\langle\lambda, Ax^{k+1}-b\rangle\cr
\le & -\frac{1}{2}\big[\|x^{k+1}-x^*\|_P^2-\|x^k-x^*\|_P^2+\|x^{k+1}-x^k\|_P^2\big]+\frac{L_f}{2}\|x^{k+1}-x^k\|^2\cr
&+\frac{1}{2\gamma}\big[\|\lambda^k-\lambda\|^2-\|\lambda^{k+1}-\lambda\|^2+\|\lambda^{k+1}-\lambda^k\|^2\big]-\frac{\beta}{\gamma^2}\|\lambda^{k+1}-\lambda^k\|^2
\end{align}
Let $\lambda=\lambda^*$ in the above inequality, and from \eqref{eq:opt-lalm}, we have
\begin{align}\label{alalm-const-bd2}
&\|x^{k+1}-x^*\|_P^2+\|x^{k+1}-x^k\|_{P-L_f I}^2+\frac{1}{\gamma}\|\lambda^{k+1}-\lambda^*\|^2+\frac{1}{\gamma}\left(\frac{2\beta}{\gamma}-1\right)\|\lambda^{k+1}-\lambda^k\|^2\cr
\le&\|x^k-x^*\|_P^2+\frac{1}{\gamma}\|\lambda^k-\lambda^*\|^2.
\end{align}
Since $P\succ L_f I$ and $\gamma<2\beta$, \eqref{alalm-const-bd2} implies the nonincreasing monotonicity of $\{\|x^k-x^*\|_P^2+\frac{1}{\gamma}\|\lambda^k-\lambda^*\|^2\}$, and thus $\{(x^k,\lambda^k)\}_{k\ge 1}$ must be bounded. Summing \eqref{alalm-const-bd2} from $k=1$ to $\infty$ gives
\begin{equation*}\sum_{k=1}^\infty \left(\|x^{k+1}-x^k\|_{P-L_f I}^2 +\frac{1}{\gamma}\big(\frac{2\beta}{\gamma}-1\big)\|\lambda^{k+1}-\lambda^k\|^2\right) <\infty,
\end{equation*}
and thus 
\begin{equation}\label{eq:diff0}
\lim_{k\to\infty}(x^{k+1},\lambda^{k+1})- (x^k,\lambda^k)=0.
\end{equation}

Let $(x^\infty,\lambda^\infty)$ be a limit point of $\{(x^k,\lambda^k)\}_{k\ge1}$ and assume the subsequence $\{(x^k,\lambda^k)\}_{k\in\cK}$ converges to it. From $Ax^{k+1}-b=\frac{1}{\gamma}(\lambda^k-\lambda^{k+1})\to 0$ as $k\to\infty$, we conclude that
\begin{equation}\label{eq:alalm-feas}Ax^\infty-b=0.
\end{equation}
In addition, letting $\cK\ni k\to\infty$ in \eqref{eq:new-x} and using \eqref{eq:diff0} gives
$$x^\infty=\argmin_x \langle \nabla f(x^\infty)-A^\top \lambda^\infty, x\rangle +g(x)+\frac{\beta}{2}\|Ax-b\|^2+\frac{1}{2}\|x-x^\infty\|_{P}^2,$$
and thus we have the optimality condition
$$0\in \nabla f(x^\infty)+\partial g(x^\infty)-A^\top \lambda^\infty+\beta A^\top(Ax^\infty-b).$$
Together with \eqref{eq:alalm-feas} implies 
$$0\in \nabla f(x^\infty)+\partial g(x^\infty)-A^\top \lambda^\infty,$$
and thus $(x^\infty,\lambda^\infty)$ satisfies the KKT conditions in \eqref{eq:kkt-lalm}. Hence, \eqref{alalm-const-bd2} still holds if $(x^*,\lambda^*)$ is replaced by $(x^\infty,\lambda^\infty)$, and we have
$$\|x^{k+1}-x^\infty\|_P^2+\frac{1}{\gamma}\|\lambda^{k+1}-\lambda^\infty\|^2
\le\|x^k-x^\infty\|_P^2+\frac{1}{\gamma}\|\lambda^k-\lambda^\infty\|^2.$$
Since $(x^\infty,\lambda^\infty)$ is a limit point of $\{(x^k,\lambda^k)\}_{k\ge 1}$, the above inequality implies the convergence of $(x^k,\lambda^k)$ to $(x^\infty,\lambda^\infty)$.

To prove \eqref{eq:alalm-const-rate}, we sum up \eqref{alalm-const-bd1} from $k=1$ through $t$ and note $P\succ L_f I$ and $\gamma<2\beta$ to have 
\begin{align*}
\sum_{k=1}^t\big[F(x^{k+1})-F(x^*)-\langle\lambda, Ax^{k+1}-b\rangle\big]
\le  \frac{1}{2}\|x^1-x^*\|_P^2+\frac{1}{2\gamma}\|\lambda^1-\lambda\|^2,
\end{align*}
which together with the convexity of $F$ implies
\begin{align}\label{alalm-const-bd3}
F(\tilde{x}^{t+1})-F(x^*)-\langle\lambda, A\tilde{x}^{t+1}-b\rangle
\le  \frac{1}{2t}\|x^1-x^*\|_P^2+\frac{1}{2\gamma t}\|\lambda^1-\lambda\|^2.
\end{align}
Noting that $\lambda^1=0$ and $x^*$ is an arbitrary optimal solution, we therefore apply Lemmas \ref{lem:x-rate} and \ref{lem:equiv-rate} with $\rho=2\|\lambda^*\|$ to complete the proof. 
\end{proof}

\subsubsection{Adaptive parameters}
In this subsection, we let the parameters $\alpha_k,\beta_k,\gamma_k$ and $P^k$ be adaptive to the iteration number $k$ and improve the previously established $O(1/t)$ convergence rate to $O(1/t^2)$, which is optimal even without the linear constraint.

\begin{theorem}\label{thm:alalm-2-t}
Under Assumptions \ref{lalm-assump1} and \ref{lalm-assump2}, let $\{(x^k,\bar{x}^k,\lambda^k)\}_{k\ge1}$ be the sequence generated from Algorithm \ref{alg:alalm} with parameters set to
\begin{equation}\label{eq:alalm-adapt-para}
\forall k:\, \alpha_k=\frac{2}{k+1},\, \gamma_k=k\gamma,\, \beta_k\ge \frac{\gamma_k}{2},\, P^k=\frac{\eta}{k} I,
\end{equation}
where $\gamma>0$ and $\eta\ge 2L_f$. 
Then
\begin{subequations}\label{eq:alalm-adapt-rate}
\begin{align}
&|F(\bar{x}^{t+1})-F(x^*)|\le  \frac{1}{t(t+1)}\left(\eta\|x^1-x^*\|^2+\frac{4\|\lambda^*\|^2}{\gamma}\right),\\
&\|A\bar{x}^{t+1}-b\|\le  \frac{1}{t(t+1)\|\lambda^*\|}\left(\eta\|x^1-x^*\|^2+\frac{4\|\lambda^*\|^2}{\gamma}\right),
\end{align}
\end{subequations}
where $(x^*,\lambda^*)$ is any point satisfying the KKT conditions in \eqref{eq:kkt-lalm}.
\end{theorem}
\begin{proof}
With the parameters given in \eqref{eq:alalm-adapt-para}, we multiply $k(k+1)$ to both sides of \eqref{eq:lalm-1step} to have
\begin{align}\label{eq:bd5}
&k(k+1)\big[F(\bar{x}^{k+1})-F(x)-\langle\lambda, A\bar{x}^{k+1}-b\rangle\big]-k(k-1)\big[F(\bar{x}^k)-F(x)-\langle\lambda, A\bar{x}^k-b\rangle\big]\cr
\le & -\eta\big[\|x^{k+1}-x\|^2-\|x^k-x\|^2+\|x^{k+1}-x^k\|^2\big]+\frac{1}{\gamma}\big[\|\lambda^k-\lambda\|^2-\|\lambda^{k+1}-\lambda\|^2+\|\lambda^{k+1}-\lambda^k\|^2\big]\cr
& -\frac{2k\beta_k}{\gamma_k^2}\|\lambda^{k+1}-\lambda^k\|^2+\frac{2k L_f}{k+1}\|x^{k+1}-x^k\|^2\cr
\le &-\eta\big[\|x^{k+1}-x\|^2-\|x^k-x\|^2\big]+\frac{1}{\gamma}\big[\|\lambda^k-\lambda\|^2-\|\lambda^{k+1}-\lambda\|^2\big].
\end{align}
Summing \eqref{eq:bd5} from $k=1$ through $t$, we have
\begin{align}\label{eq:bd6}
t(t+1)\big[F(\bar{x}^{t+1})-F(x)-\langle\lambda, A\bar{x}^{t+1}-b\rangle\big]\le  \eta \|x^1-x\|^2+\frac{1}{\gamma}\|\lambda^1-\lambda\|^2.
\end{align}
Letting $x=x^*$ in the above inequality and then applying Lemmas \ref{lem:x-rate} and \ref{lem:equiv-rate}, we obtain the desired result.
\end{proof}

\begin{remark}
With a positive definite matrix $P^k$, the subproblem \eqref{eq:new-x} becomes strongly convex and thus has a unique solution. One drawback of Theorem \ref{thm:alalm-2-t} is that the setting in \eqref{eq:alalm-adapt-para} does not allow linearization to the augmented term. The coexistence of the possibly nonsmooth term $g$ and the augmented term $\|Ax-b\|^2$ can still cause difficult subproblems. In that case, we can solve the subproblem inexactly. Theoretically we will lose the fast convergence  shown in Theorem \ref{thm:alalm-2-t}. 
However, empirically we still observe fast convergence even subproblems are solved to a medium accuracy; see the experimental results in section \ref{sec:quadprog}. To linearize the augmented term and retain $O(1/t^2)$ convergence, we need assume strong convexity of the objective; see Theorem \ref{thm:apadmm-rate} below. 
\end{remark}

\subsection{Analysis of the accelerated linearized ADMM}
In this subsection, we establish the convergence rate of Algorithm \ref{alg:apadmm}. In addition to Assumption \ref{lalm-assump1}, we make the following assumptions to the objective function of \eqref{eq:prob-padmm}.
\begin{assumption}\label{assump3}
The function $h$ has Lipschitz continuous gradient with constant $L_h$, and $g$ and $h$ are strongly convex with modulus $\mu_g$ and $\mu_h$ that satisfy $\mu_g+\mu_h>0$.
\end{assumption}
 Note that without strong convexity, $O(1/t)$ convergence rate can be shown; see \cite{ouyang2015accelerated, gao2015first} for example. Also note that the $O(1/t^2)$ rate has been established in \cite{goldstein2014fast} if both $f$ and $g+h$ are strongly convex and no linearization is performed. 

Similar to the analysis in the previous subsection, we first establish a result of running one iteration of Algorithm \ref{alg:apadmm}.
\begin{lemma}[One-iteration result]\label{lem:padmm-1step}
Let $\{(y^k,z^k,\lambda^k)\}_{k\ge 1}$ be the sequence generated from Algorithm \ref{alg:apadmm}. Then for any $(y,z,\lambda)$ such that $By+Cz=b$, it holds
\begin{align}\label{eq:2b-bd4}
& F(y^{k+1},z^{k+1})-F(y,z)-\langle\lambda, By^{k+1}+Cz^{k+1}-b\rangle\cr
\le &-\left\langle \frac{1}{\gamma_k}(\lambda^k-\lambda^{k+1}), \lambda-\lambda^k+\frac{\beta_k}{\gamma_k}(\lambda^k-\lambda^{k+1})\right\rangle+ \beta_k\left\langle \frac{1}{\gamma_k}(\lambda^k-\lambda^{k+1})-C(z^{k+1}-z),C(z^{k+1}-z^k)\right\rangle\cr
&+\frac{L_h}{2}\|z^{k+1}-z^k\|^2-\frac{\mu_h}{2}\|z^k-z\|^2-\frac{\mu_g}{2}\|z^{k+1}-z\|^2\cr
&-\langle y^{k+1}-y,P^k(y^{k+1}-y^k)\rangle-\langle z^{k+1}-z,Q^k(z^{k+1}-z^k)\rangle,
\end{align}
where $F$ is given in \eqref{eq:F-padmm}.
\end{lemma}
\begin{proof}
From the update \eqref{eq:apadmm-y}, we have the optimality condition
$$0= \tilde{\nabla} f(y^{k+1})-B^\top\lambda^k+\beta_k B^\top(By^{k+1}+Cz^k-b)+P^k(y^{k+1}-y^k),$$
where $\tilde{\nabla} f(y^{k+1})$ is a subgradient of $f$ at $y^{k+1}$.
Thus for any $y$,
\begin{align}\label{eq:2b-bd1}
0=&\left\langle y^{k+1}-y, \tilde{\nabla} f(y^{k+1})-B^\top\lambda^k+\beta_k B^\top(By^{k+1}+Cz^k-b)+P^k(y^{k+1}-y^k)\right\rangle\cr
\ge&f(y^{k+1})-f(y)+\left\langle y^{k+1}-y, -B^\top\lambda^k+\beta_k B^\top(By^{k+1}+Cz^k-b)+P^k(y^{k+1}-y^k)\right\rangle\cr
= & f(y^{k+1})-f(y)+\left\langle y^{k+1}-y, -B^\top\lambda^k+\beta_k B^\top(By^{k+1}+Cz^{k+1}-b)-\beta_k B^\top C(z^{k+1}-z^k)\right\rangle\cr
&+\big\langle y^{k+1}-y,P^k(y^{k+1}-y^k)\big\rangle\cr
=&f(y^{k+1})-f(y)+\left\langle B(y^{k+1}-y), -\lambda^k+\frac{\beta_k}{\gamma_k}(\lambda^k-\lambda^{k+1})\right\rangle - \beta_k\big\langle B(y^{k+1}-y),C(z^{k+1}-z^k)\big\rangle\cr
&+\big\langle y^{k+1}-y,P^k(y^{k+1}-y^k)\big\rangle,
\end{align}
where in the last equality, we have used the update rule \eqref{eq:apadmm-lam}. 
Similar to \eqref{eq:bd1}, we have
\begin{equation}\label{eq:2b-bd2}
h(z^{k+1})\le h(z)+\langle\nabla h(z^k), z^{k+1}-z\rangle+\frac{L_h}{2}\|z^{k+1}-z^k\|^2-\frac{\mu_h}{2}\|z^k-z\|^2.
\end{equation}

From the update rule of $z^{k+1}$, we have the optimality condition:
$$0= \tilde{\nabla} g(z^{k+1})+\nabla h(z^k)-C^\top \lambda^k +\beta_k C^\top(Bx^{k+1}+Cz^{k+1}-b)+Q^k(z^{k+1}-z^k).$$
Hence, for any $z$, it holds
\begin{align}\label{eq:2b-bd3}
0=&\left\langle z^{k+1}-z, \tilde{\nabla} g(z^{k+1})+\nabla h(z^k)-C^\top \lambda^k +\beta_k C^\top(By^{k+1}+Cz^{k+1}-b)+Q^k(z^{k+1}-z^k)\right\rangle\cr
\ge & g(z^{k+1})-g(z)+\frac{\mu_g}{2}\|z^{k+1}-z\|^2 + \langle z^{k+1}-z, \nabla h(z^k)\rangle \cr
&+\left\langle z^{k+1}-z,-C^\top \lambda^k +\beta_k C^\top(By^{k+1}+Cz^{k+1}-b)+Q^k(z^{k+1}-z^k)\right\rangle\cr
= &g(z^{k+1})-g(z)+\frac{\mu_g}{2}\|z^{k+1}-z\|^2+\big\langle z^{k+1}-z, \nabla h(z^k)\big\rangle +\big\langle z^{k+1}-z,Q^k(z^{k+1}-z^k)\big\rangle\cr
& +\left\langle C(z^{k+1}-z), -\lambda^k+\frac{\beta_k}{\gamma_k}(\lambda^k-\lambda^{k+1})\right\rangle,
\end{align}
where the inequality follows from the convexity of $g$. 

Since $(y,z)$ is feasible, summing \eqref{eq:2b-bd1}, \eqref{eq:2b-bd2} and \eqref{eq:2b-bd3} gives
\begin{align*}
& F(y^{k+1},z^{k+1})-F(y,z)-\langle\lambda, By^{k+1}+Cz^{k+1}-b\rangle\cr
\le & -\left\langle B(y^{k+1}-y), -\lambda^k+\frac{\beta_k}{\gamma_k}(\lambda^k-\lambda^{k+1})\right\rangle -\left\langle C(z^{k+1}-z), -\lambda^k+\frac{\beta_k}{\gamma_k}(\lambda^k-\lambda^{k+1})\right\rangle\cr
&-\langle\lambda, By^{k+1}+Cz^{k+1}-b\rangle+ \beta_k\langle B(y^{k+1}-y),C(z^{k+1}-z^k)\rangle\cr
&+\frac{L_h}{2}\|z^{k+1}-z^k\|^2-\frac{\mu_h}{2}\|z^k-z\|^2-\frac{\mu_g}{2}\|z^{k+1}-z\|^2\cr
&-\langle y^{k+1}-y,P^k(y^{k+1}-y^k)\rangle-\langle z^{k+1}-z,Q^k(z^{k+1}-z^k)\rangle
\end{align*}
which implies \eqref{eq:2b-bd4} by noting the update rule \eqref{eq:apadmm-lam}.
\end{proof}

When constant parameters are used in Algorithm \ref{alg:apadmm}, one can sum up \eqref{eq:2b-bd4} from $k=1$ through $t$ and use \eqref{uv-cross} to show an $O(1/t)$ convergence result. This has already been established in the literature; see \cite{gao2015first} for example. Hence, we state the result here without proof, and note that the result does not require any strong convexity of the objective.

\begin{theorem}\label{thm:apadmm-rate-const}
Assume the existence of $(x^*,\lambda^*)=(y^*,z^*,\lambda^*)$ satisfying \eqref{eq:kkt-lalm} and the gradient Lipschitz continuity of $h$. Let $\{(y^k,z^k,\lambda^k)\}_{k\ge 1}$ be the sequence generated from Algorithm \ref{alg:apadmm} with parameters set to
\begin{equation}\label{eq:para-apadmm-const}
\beta_k=\gamma_k=\gamma>0,\, P^k=P\succeq 0,\, Q^k=Q\succeq L_h I,\,\forall k.
\end{equation}
Then 
\begin{align*}
&\big|F(\tilde{y}^{t+1},\tilde{z}^{t+1})-F(y^*,z^*)\big|\le \frac{1}{2t}\left(\frac{4\|\lambda^*\|^2}{\gamma}+\|y^1-y^*\|_P^2+\|z^1-z^*\|_{Q+C^\top C}^2\right)\\
&\|B\tilde{y}^{t+1}+C\tilde{z}^{t+1}-b\|\le\frac{1}{2t\|\lambda^*\|}\left(\frac{4\|\lambda^*\|^2}{\gamma}+\|y^1-y^*\|_P^2+\|z^1-z^*\|_{Q+C^\top C}^2\right),
\end{align*}
where
$$\tilde{y}^{t+1}=\frac{\sum_{k=1}^t y^{k+1}}{t},\,\tilde{z}^{t+1}=\frac{\sum_{k=1}^t z^{k+1}}{t}.$$
\end{theorem}

Adapting the parameters, we can accelerate the rate to $O(1/t^2)$ as shown below.
\begin{theorem}\label{thm:apadmm-rate}
Under Assumptions \ref{lalm-assump1} and \ref{assump3}, let $\{(y^k,z^k,\lambda^k)\}_{k\ge 1}$ be the sequence generated from Algorithm \ref{alg:apadmm} with parameters set to
\begin{subequations}\label{eq:2b-para}
\begin{align}
&\beta_k=\gamma_k=(k+1)\gamma,\,\forall k\ge 1,\label{eq:2b-para-beta}\\
& P^k = \frac{P}{k+1}I,\,\forall k\ge 1,\label{eq:2b-para-P}\\
& Q^k=(k+1)\big(Q-\gamma C^\top C\big)+ L_h I,\,\forall k\ge 1, \label{eq:2b-para-Q}
\end{align}
\end{subequations}
where $P\succeq 0$ and
$\eta\gamma C^\top C \preceq Q\preceq \frac{\mu_g+\mu_h}{2}I$ with $\eta\ge 1$.   
Let 
\begin{equation}\label{eq:2b-def-k0}
k_0=\left\lceil1+\frac{2(L_h-\mu_h)}{\mu_g+\mu_h}\right\rceil.
\end{equation}
Then we have
\begin{equation}\label{eq:rate-z}
\|z^k-z^*\|^2_Q\le \frac{2 \phi_1(y^*,z^*,\lambda^*)}{k(k+k_0)},\quad\|z^k-z^*\|^2\le \frac{2 \phi_1(y^*,z^*,\lambda^*)}{(k+k_0)(L_h+\mu_h+2\mu_g)},
\end{equation}
and
\begin{subequations}\label{eq:rate-2-t-2b}
\begin{align}
&|F(\tilde{y}^{t+1},\tilde{z}^{t+1})-F(y^*,z^*)|\le\frac{2}{t(t+2k_0+3)}\phi_1(y^*,z^*,2\lambda^*)\label{eq:rate-obj-2-t-2b}\\
&\|B\tilde{y}^{t+1}+C\tilde{z}^{t+1}-b\|\le \frac{2}{t(t+2k_0+3)\|\lambda^*\|}\phi_1(y^*,z^*,2\lambda^*)\label{eq:rate-fea-2-t-2b}
\end{align}
\end{subequations}
where 
$$\tilde{y}^{t+1}=\frac{\sum_{k=1}^t(k+k_0+1)y^{k+1}}{\sum_{k=1}^t(k+k_0+1)},\quad \tilde{z}^{t+1}=\frac{\sum_{k=1}^t(k+k_0+1)z^{k+1}}{\sum_{k=1}^t(k+k_0+1)},$$
 and
\begin{align}\label{eq:def-phi}
\phi_k(y,z,\lambda)=&\frac{k+k_0}{2k}\|y^k-y\|_P^2+\frac{k+k_0}{2}\left(k\|z^k-z\|_Q^2+(L_h+\mu_g)\|z^k-z\|^2\right)+\frac{k+k_0}{2\gamma k}\|\lambda-\lambda^k\|^2.
\end{align}
In addition, if $P\succ 0$ and $\eta>1$, then $\{(y^k,z^k,\lambda^k)\}_{k\ge1}$ is bounded,  and 
\begin{subequations}
\begin{align}
&\|By^{k+1}+Cz^{k+1}-b\|\le o\left(\frac{1}{k+1}\right),\label{eq:rate-fea-2b}\\
&|F(y^{k+1},z^{k+1})-F(y^*,z^*)|\le O\left(\frac{1}{k+1}\right).\label{eq:rate-obj-2b}
\end{align}
\end{subequations}
\end{theorem}

\begin{remark}
Note that if $Q$ is a diagonal matrix in \eqref{eq:2b-para-Q}, then the augmented term in \eqref{eq:apadmm-z} is also linearized. If $f=0$ and $B=0$, the problem \eqref{eq:prob-padmm} reduces to \eqref{eq:prob-lalm}. Therefore, Theorem \ref{thm:apadmm-rate} implies that we can further linearize the augmented term in the subproblem of the linearized ALM and still obtain $O(1/t^2)$ convergence if the objective is strongly convex.

Also note that taking $P=0$ and $Q=\gamma C^\top C$ leads to the standard ADMM with adaptive parameters. Hence, we obtain the same order of convergence rate as that in \emph{\cite{goldstein2014fast}} with strictly weaker conditions.
\end{remark}

To show this theorem, we first establish a few inequalities.
\begin{proposition}\label{prop1}
Let $k_0$ be defined in \eqref{eq:2b-def-k0}. Then for any $k\ge 1$,
\begin{align}\label{eq:2b-cond-k0}
(k+k_0)\big(k Q+(L_h+\mu_g)I\big)\succeq (k+k_0+1)\big((k+1)Q+(L_h-\mu_h)I\big).
\end{align}
\end{proposition}
\begin{proof}
Expanding the left hand side of the inequality and using $Q\preceq\frac{\mu_h+\mu_g}{2}I$ and \eqref{eq:2b-def-k0} shows the result.
\end{proof}

\begin{proposition}\label{prop2}
Under the assumptions of Theorem \ref{thm:apadmm-rate}, we have
\begin{align}\label{eq:2b-bd5}
& F(y^{k+1},z^{k+1})-F(y,z)-\langle\lambda, By^{k+1}+Cz^{k+1}-b\rangle\cr
\le & -\frac{1}{2\gamma(k+1)}\big[\|\lambda-\lambda^{k+1}\|^2-\|\lambda-\lambda^k\|^2\big]-\frac{\eta-1}{2\eta\gamma(k+1)}\|\lambda^k-\lambda^{k+1}\|^2\\
&-\frac{1}{2(k+1)}\big[\|y^{k+1}-y\|_{P}^2-\|y^k-y\|_{P}^2+\|y^{k+1}-y^k\|_{P}^2\big]
\cr
&-\frac{1}{2}\left((k+1)\|z^{k+1}-z\|_Q^2+(L_h+\mu_g)\|z^{k+1}-z\|^2\right)+\frac{1}{2}\left((k+1)\|z^k-z\|_Q^2+(L_h-\mu_h)\|z^k-z\|^2\right).\nonumber
\end{align}
\end{proposition}

\begin{proof}
Since $\beta_k=\gamma_k$, we use \eqref{uv-cross} and have from \eqref{eq:2b-bd4} that 
\begin{align}\label{eq:2b-bd4-2}
& F(y^{k+1},z^{k+1})-F(y,z)-\langle\lambda, By^{k+1}+Cz^{k+1}-b\rangle\cr
\le & -\frac{1}{2\gamma_k}\big[\|\lambda^k-\lambda^{k+1}\|^2+\|\lambda-\lambda^{k+1}\|^2-\|\lambda-\lambda^k\|^2\big]+ \left\langle \lambda^k-\lambda^{k+1},C(z^{k+1}-z^k)\right\rangle\cr
&-\frac{\gamma_k}{2}\big[\|C(z^{k+1}-z)\|^2-\|C(z^k-z)\|^2+\|C(z^{k+1}-z^k)\|^2\big]+\frac{L_h}{2}\|z^{k+1}-z^k\|^2-\frac{\mu_h}{2}\|z^k-z\|^2\cr
&-\frac{\mu_g}{2}\|z^{k+1}-z\|^2-\frac{1}{2}\big[\|y^{k+1}-y\|_{P^k}^2-\|y^k-y\|_{P^k}^2+\|y^{k+1}-y^k\|_{P^k}^2\big]\cr
&-\frac{1}{2}\big[\|z^{k+1}-z\|_{Q^k}^2-\|z^k-z\|_{Q^k}^2+\|z^{k+1}-z^k\|_{Q^k}^2\big].
\end{align}
Note that from the parameter setting, we have
\begin{align}\label{eq:2b-bd5-4}
&\left\langle \lambda^k-\lambda^{k+1},C(z^{k+1}-z^k)\right\rangle-\frac{\gamma_k}{2}\|C(z^{k+1}-z^k)\|^2+\frac{L_h}{2}\|z^{k+1}-z^k\|^2-\frac{1}{2}\|z^{k+1}-z^k\|_{Q^k}^2\cr
=& \left\langle \lambda^k-\lambda^{k+1},C(z^{k+1}-z^k)\right\rangle-\frac{k+1}{2}\|z^{k+1}-z^k\|_{Q}^2\cr
\le & \left\langle \lambda^k-\lambda^{k+1},C(z^{k+1}-z^k)\right\rangle-\frac{\eta\gamma_k}{2}\|z^{k+1}-z^k\|_{C^\top C}^2\cr
\le & \frac{1}{2\eta\gamma_k}\|\lambda^k-\lambda^{k+1}\|^2.
\end{align}
Plugging \eqref{eq:2b-bd5-4} and also the parameters in \eqref{eq:2b-para} into \eqref{eq:2b-bd4-2} gives \eqref{eq:2b-bd5}.
\end{proof}

Now we are ready to show Theorem \ref{thm:apadmm-rate}.
\begin{proof} [Proof of Theorem \ref{thm:apadmm-rate}]

Letting $(y,z)=(y^*,z^*)$ in \eqref{eq:2b-bd5} and rearranging terms gives
\begin{align}
& \big[F(y^{k+1},z^{k+1})-F(y^*,z^*)-\langle\lambda, By^{k+1}+Cz^{k+1}-b\rangle\big]+\frac{1}{2(k+1)}\|y^*-y^{k+1}\|_P^2\cr
&+\frac{1}{2}\left((k+1)\|z^{k+1}-z^*\|_Q^2+(L_h+\mu_g)\|z^{k+1}-z^*\|^2\right)+\frac{1}{2\gamma(k+1)}\|\lambda-\lambda^{k+1}\|^2\cr
\le & \frac{1}{2(k+1)}\|y^k-y^*\|_P^2+\frac{1}{2}\left((k+1)\|z^k-z^*\|_Q^2+(L_h-\mu_h)\|z^k-z^*\|^2\right)\cr
&+\frac{1}{2\gamma(k+1)}\|\lambda-\lambda^k\|^2-\frac{\eta-1}{2\eta\gamma(k+1)}\|\lambda^k-\lambda^{k+1}\|^2.\label{eq:2b-bd8}
\end{align}
Multiplying $k+k_0+1$ to both sides of the above inequality and using notation $\phi_k$ defined in \eqref{eq:def-phi}, we have 
\begin{align}
& (k+k_0+1)\big[F(y^{k+1},z^{k+1})-F(y^*,z^*)-\langle\lambda, By^{k+1}+Cz^{k+1}-b\rangle\big]+\phi_{k+1}(y^*,z^*,\lambda)\cr
\le & \frac{k+k_0+1}{2(k+1)}\|y^k-y^*\|_P^2+\frac{k+k_0+1}{2}\left((k+1)\|z^k-z^*\|_Q^2+(L_h-\mu_h)\|z^k-z^*\|^2\right)\cr
&+\frac{k+k_0+1}{2\gamma(k+1)}\left(\|\lambda-\lambda^k\|^2-\frac{\eta-1}{\eta}\|\lambda^k-\lambda^{k+1}\|^2\right)\cr
\le & \frac{k+k_0}{2k}\|y^k-y^*\|_P^2+\frac{k+k_0}{2}\left(k\|z^k-z^*\|_Q^2+(L_h+\mu_g)\|z^k-z^*\|^2\right)+\frac{k+k_0}{2\gamma k}\|\lambda-\lambda^k\|^2\cr
&-\frac{k+k_0+1}{2\gamma(k+1)}\frac{\eta-1}{\eta}\|\lambda^k-\lambda^{k+1}\|^2,\cr
=&\phi_k(y^*,z^*,\lambda)-\frac{k+k_0+1}{2\gamma(k+1)}\frac{\eta-1}{\eta}\|\lambda^k-\lambda^{k+1}\|^2,\label{eq:2b-bd9}
\end{align}
where in the second inequality, we have used \eqref{eq:2b-cond-k0} and the decreasing monotonicity of $\frac{k+k_0+1}{k+1}$ with respect to $k$.

Letting $\lambda=\lambda^*$ in \eqref{eq:2b-bd9} and using \eqref{eq:opt-lalm}, we have 
\begin{equation}\label{eq:dec-phi}
\phi_{k+1}(y^*,z^*,\lambda^*)\le\phi_k(y^*,z^*,\lambda^*).
\end{equation}
In addition, note that
\begin{align*}
&F(y^{k+1},z^{k+1})-F(y^*,z^*)-\langle\lambda^*, By^{k+1}+Cz^{k+1}-b\rangle\\
=& F(y^{k+1},z^{k+1})-F(y^*,z^*)-\langle\lambda^*, B(y^{k+1}-y^*)+C(z^{k+1}-z^*)\rangle\\
=&f(y^{k+1})-f(x^*)-\langle B^\top\lambda^*,y^{k+1}-y^*\rangle+(g+h)(z^{k+1})-(g+h)(z^*)-\langle C^\top\lambda^*, z^{k+1}-z^*\rangle\\
\ge&\frac{\mu_g+\mu_h}{2}\|z^{k+1}-z^*\|^2,
\end{align*} 
where the inequality is from the convexity of $f$ and $g+h$ and also the KKT conditions in \eqref{eq:kkt-lalm}. Hence, from \eqref{eq:2b-bd9} and \eqref{eq:dec-phi}, it follows that
$$\frac{(\mu_g+\mu_h)(k+k_0+1)}{2}\|z^{k+1}-z^*\|^2+\phi_{k+1}(y^*,z^*,\lambda^*)\le \phi_1(y^*,z^*,\lambda^*),$$
and thus we obtain the results in \eqref{eq:rate-z}. If $P\succ0$, the above inequality indicates the boundedness of $\{(x^k,y^k,\lambda^k)\}$ 

Again, letting $\lambda=\lambda^*$ in \eqref{eq:2b-bd9} and summing it from $k=1$ through $t$, we conclude from \eqref{eq:opt-lalm} and \eqref{eq:dec-phi} that
$$\sum_{k=1}^t\frac{k+k_0+1}{2\gamma(k+1)}\frac{\eta-1}{\eta}\|\lambda^k-\lambda^{k+1}\|^2\le\phi_1(y^*,z^*,\lambda^*),$$
and thus letting $t\to\infty$, we have $\lambda^k-\lambda^{k+1}\to 0$ from the above inequality as $\eta>1$, and thus \eqref{eq:rate-fea-2b} follows from the update rule \eqref{eq:apadmm-lam}. 
Furthermore, from the boundedness of $\{(y^k,z^k,\lambda^k)\}$, we let $\lambda=0$ in \eqref{eq:2b-bd8} to have 
$F(y^{k+1},z^{k+1})-F(y^*,z^*)\le O\left(\frac{1}{k+1}\right).$
Using \eqref{eq:opt-lalm} and \eqref{eq:rate-fea-2b}, we have
$F(y^{k+1},z^{k+1})-F(y^*,z^*)\ge -O\left(\frac{1}{k+1}\right),$
and thus \eqref{eq:rate-obj-2b} follows.

Finally, summing \eqref{eq:2b-bd9} from $k=1$ through $t$ and noting $\phi_k\ge0,\forall k$, we have
\begin{align*} \sum_{k=1}^t(k+k_0+1)\big[F(y^{k+1},z^{k+1})-F(y^*,z^*)-\langle\lambda, By^{k+1}+Cz^{k+1}-b\rangle\big]
\le \phi_1(y^*,z^*,\lambda).
\end{align*}
Then by the convexity of $F$, we have from the above inequality that for any $\lambda$,
\begin{align*}
F(\tilde{y}^{t+1},\tilde{z}^{t+1})-F(y^*,z^*)-\langle \lambda, B\tilde{y}^{t+1}+C\tilde{z}^{t+1}-b\rangle
\le \frac{\phi_1(y^*,z^*,\lambda)}{\sum_{k=1}^t(k+k_0+1)}
\end{align*}
By Lemmas \ref{lem:x-rate} and \ref{lem:equiv-rate} and the initialization $\lambda^1=0$, the above result implies the desired results in \eqref{eq:rate-2-t-2b}. This completes the proof.
\end{proof}


\section{Numerical results}\label{sec:numerical}
In this section, we test the proposed accelerated methods on solving three problems: quadratic programming, total variation regularized image denoising, and elastic net regularized support vector machine. We compare them to nonaccelerated methods and also existing accelerated methods to demonstrate their efficiency. 
\subsection{Quadratic programming}\label{sec:quadprog}
In this subsection, we test Algorithm \ref{alg:alalm} on quadratic programming. First, we compare the algorithm with fixed and adaptive parameters, i.e., nonaccelerated ALM and accelerated ALM, on equality constrained quadratic programming (ECQP):
\begin{equation}\label{eq:eq-qp}
\min_x~F(x)=\frac{1}{2}x^\top Q x + c^\top x,\st Ax=b.
\end{equation}
Note that ECQP can be solved in a direct way by solving a linear equation (c.f., \cite[Section 16.1]{nocedal2006numerical}), so ALM may not be the best choice for \eqref{eq:eq-qp}. Our purpose of using this simple example is to validate acceleration. 

We set the problem size to $m=20, n=500$ and generate $A\in\RR^{m\times n},b, c$ and $Q\in\RR^{n\times n}$ according to standard Gaussian distribution, where $Q$ is made to be a positive definite matrix. We set the parameters of Algorithm \ref{alg:alalm} to $\alpha_k=1, \beta_k=\gamma_k=m$ and $P^k= \|Q\|_2 I,\,\forall k$ for the nonaccelerated ALM, and $\alpha_k=\frac{2}{k+1}, \beta_k=\gamma_k=mk$ and $P^k=\frac{2\|Q\|_2}{k}I,\,\forall k$ for the accelerated ALM. Figure \ref{fig:qp-exact} plots the objective distance to the optimal value $|F(x)-F(x^*)|$ and the violation of feasibility $\|Ax-b\|$ given by the two methods. We can see that Algorithm \ref{alg:alalm} with adaptive parameters performs significantly better than it with fixed parameters, in both objective and feasibility measures.

\begin{figure}
\begin{center}
\begin{tabular}{cc}
\includegraphics[width=0.25\textwidth]{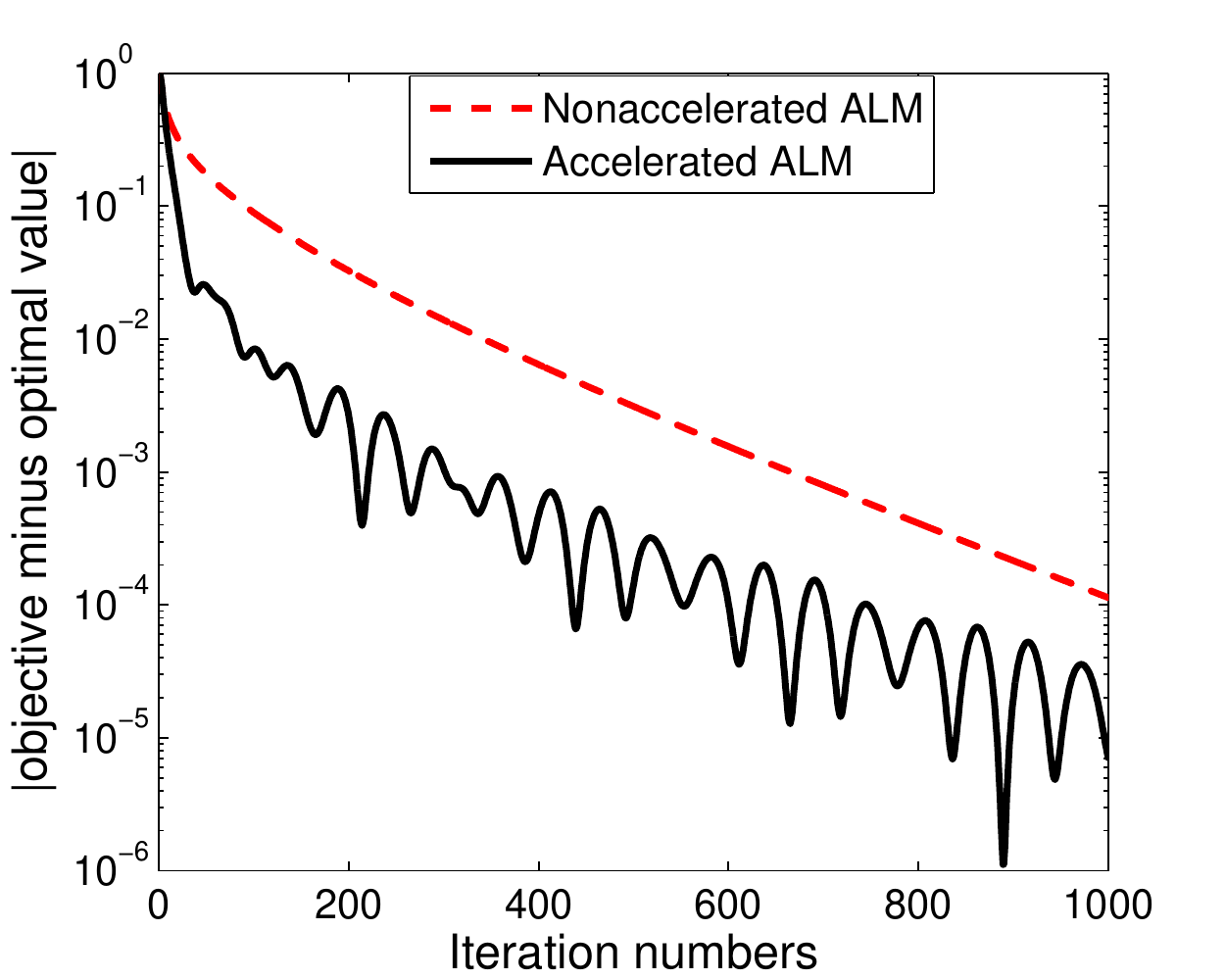} &
\includegraphics[width=0.25\textwidth]{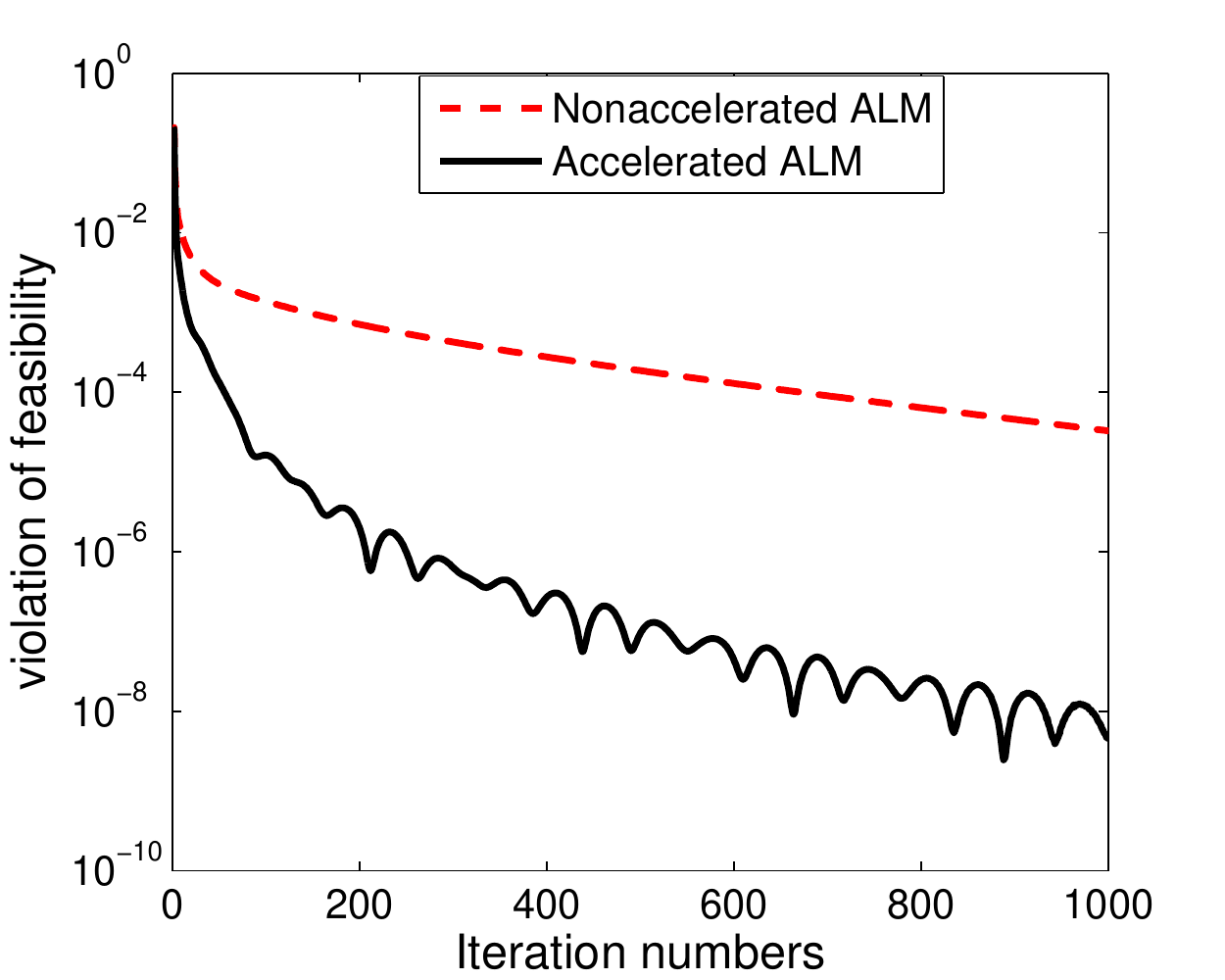}
\end{tabular}
\end{center}
\caption{Results by the nonaccelerated ALM (Algorithm \ref{alg:alalm} with fixed parameters) and the accelerated ALM (Algorithm \ref{alg:alalm} with adaptive parameters) on solving \eqref{eq:eq-qp}. Left: the distance of the objective value to the optimal value $|F(x)-F(x^*)|$; Right: the violation of feasibility $\|Ax-b\|$.}\label{fig:qp-exact}
\end{figure}

Secondly, we test the accelerated ALM on the nonnegative linearly constrained quadratic programming, which is formulated as follows:
\begin{equation}\label{eq:quadprog}
\min_x ~F(x)=\frac{1}{2}x^\top Q x + c^\top x,
\st  Ax = b,\ x\ge0.
\end{equation}

In the test, we set the problem size to $m=50$ and $n=1000$. We let $Q=HH^\top$, where $H\in\RR^{n\times (n-100)}$ and is generated according to standard Gaussian distribution. Hence, the objective is only weakly convex. The elements of $b$ and $c$ follow identically independent uniform distribution and standard Gaussian distribution, respectively. Thus, $b\ge 0$. The matrix $A\in\RR^{m\times n}$ has the form of $[B,I]$ to make sure feasibility of the problem. We generate $B$ according to both Gaussian and uniform distribution. Note that the uniformly distributed $B$ leads to more difficult problem. 

We set the parameters of Algorithm \ref{alg:alalm} according to \eqref{eq:alalm-adapt-para} with $\gamma=m$, $\eta=2\|Q\|_2$, and $\beta_k=\gamma_k,\,\forall k$. The most difficult step in Algorithm \ref{alg:alalm} is \eqref{eq:new-x}, which does not have a closed form solution with the above setting. We solve the subproblem by the interior-point method to a tolerance \verb|subtol|. Since $A$ only has 50 rows, each step of the interior-point method only needs to solve a $50\times 50$ equation and do some componentwise multiplication. We notice that ALALM converges fast in the beginning but slows down as it approaches the solution. Hence, we also test to restart it after a fixed number of iterations, and in this test, we simply restart it every 50 iterations.

We compare ALALM to FISTA \cite{FISTA2009}, which also has $O(1/t^2)$ convergence rate. At each iteration, FISTA requires a projection to the constraint set of \eqref{eq:quadprog}, and we solve it also by the interior-point method to the tolerance \verb|subtol|. Again, each step of the interior-point method only needs to solve a $50\times 50$ equation and do some componentwise multiplication. We also test restarted FISTA by restarting it every 50 iterations. Note that a restarted FISTA is proposed in \cite{o2015adaptive} by checking the monotonicity of the objective value or gradient norm. However, since subproblems are solved inaccurately, the restart scheme in \cite{o2015adaptive} does not work here.

Figure \ref{fig:qp-G50} plots the results corresponding to Gaussian randomly generated matrix $B$ and Figure \ref{fig:qp-U50} corresponding to uniformly random $B$. In both figures, \verb|subtol| varies among $\{10^{-6},10^{-8},10^{-10}\}$. From the figures, we see that both FISTA and ALALM perform better when restarted periodically, and ALALM performs more stably than FISTA to different \verb|subtol|. Even if the subproblems are solved inaccurately only to the tolerance $10^{-6}$, the restarted ALALM can still reach almost machine accuracy. However, FISTA can reach an accurate solution only if the subproblems are solved to a high accuracy such as $\verb|subtol|=10^{-10}$ and $B$ is Gaussian randomly generated.

\begin{figure}
\begin{center}
\begin{tabular}{ccc}
$\verb|subtol|=10^{-6}$ & $\verb|subtol|=10^{-8}$ & $\verb|subtol|=10^{-10}$\\
\includegraphics[width=0.25\textwidth]{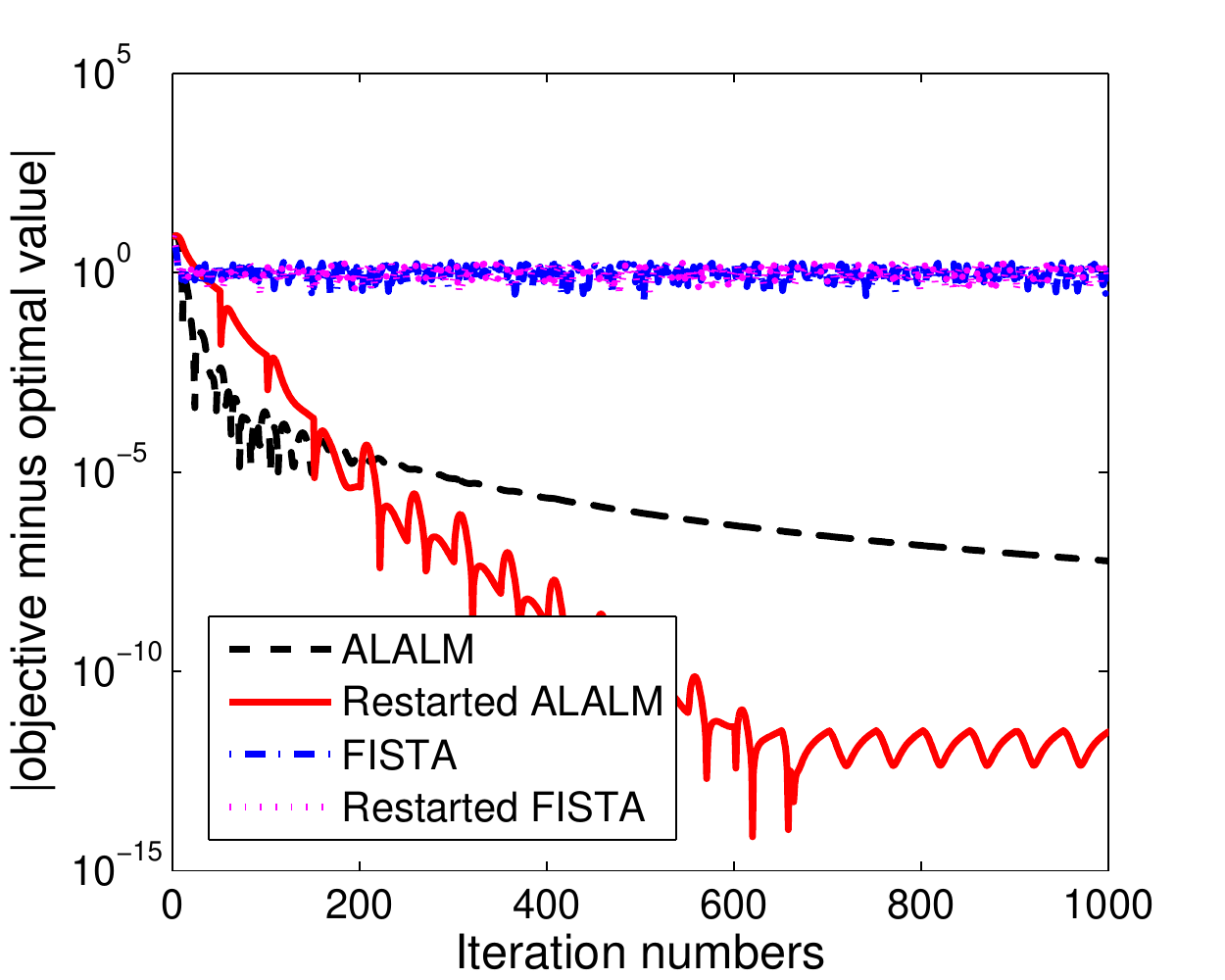} &
\includegraphics[width=0.25\textwidth]{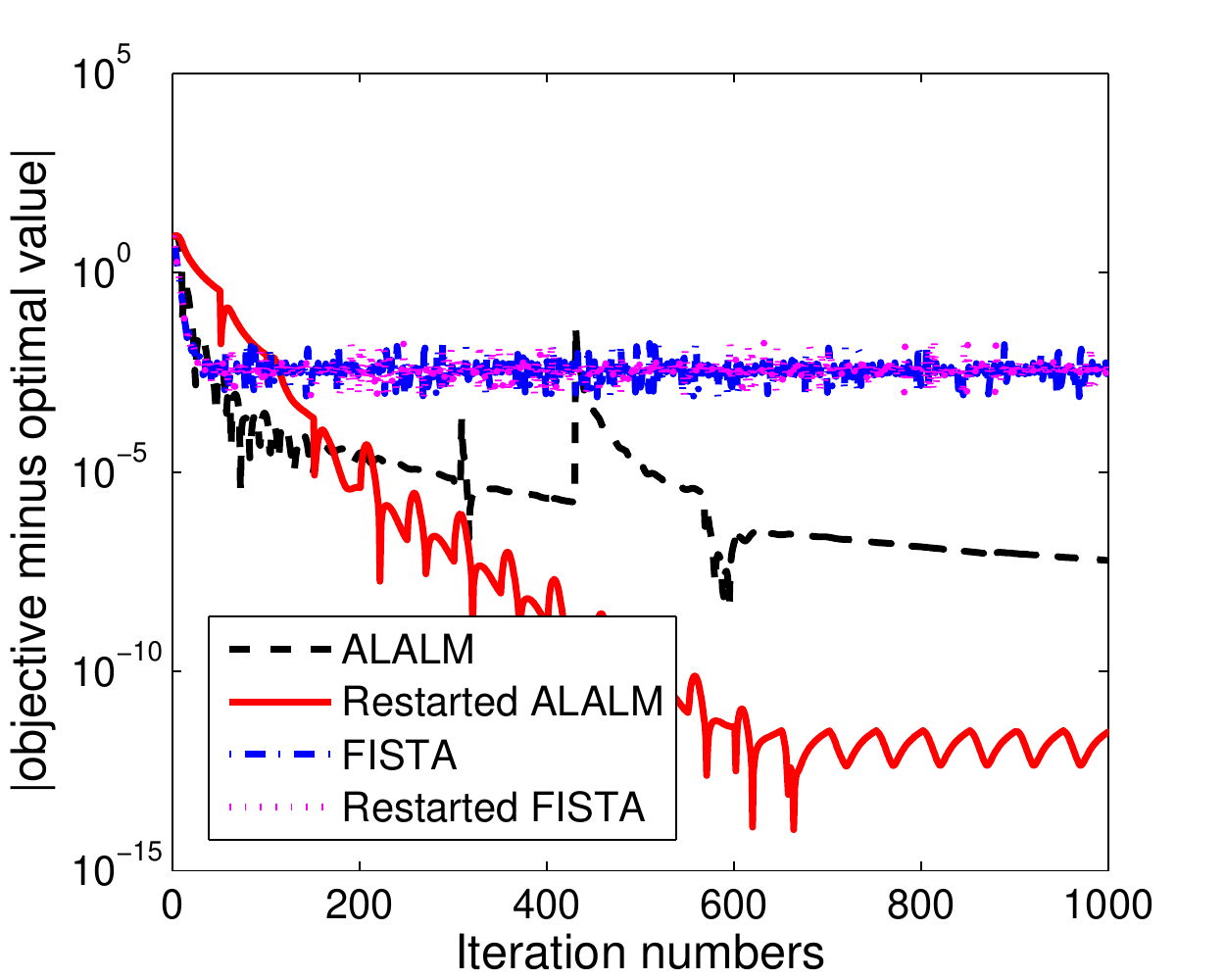} &
\includegraphics[width=0.25\textwidth]{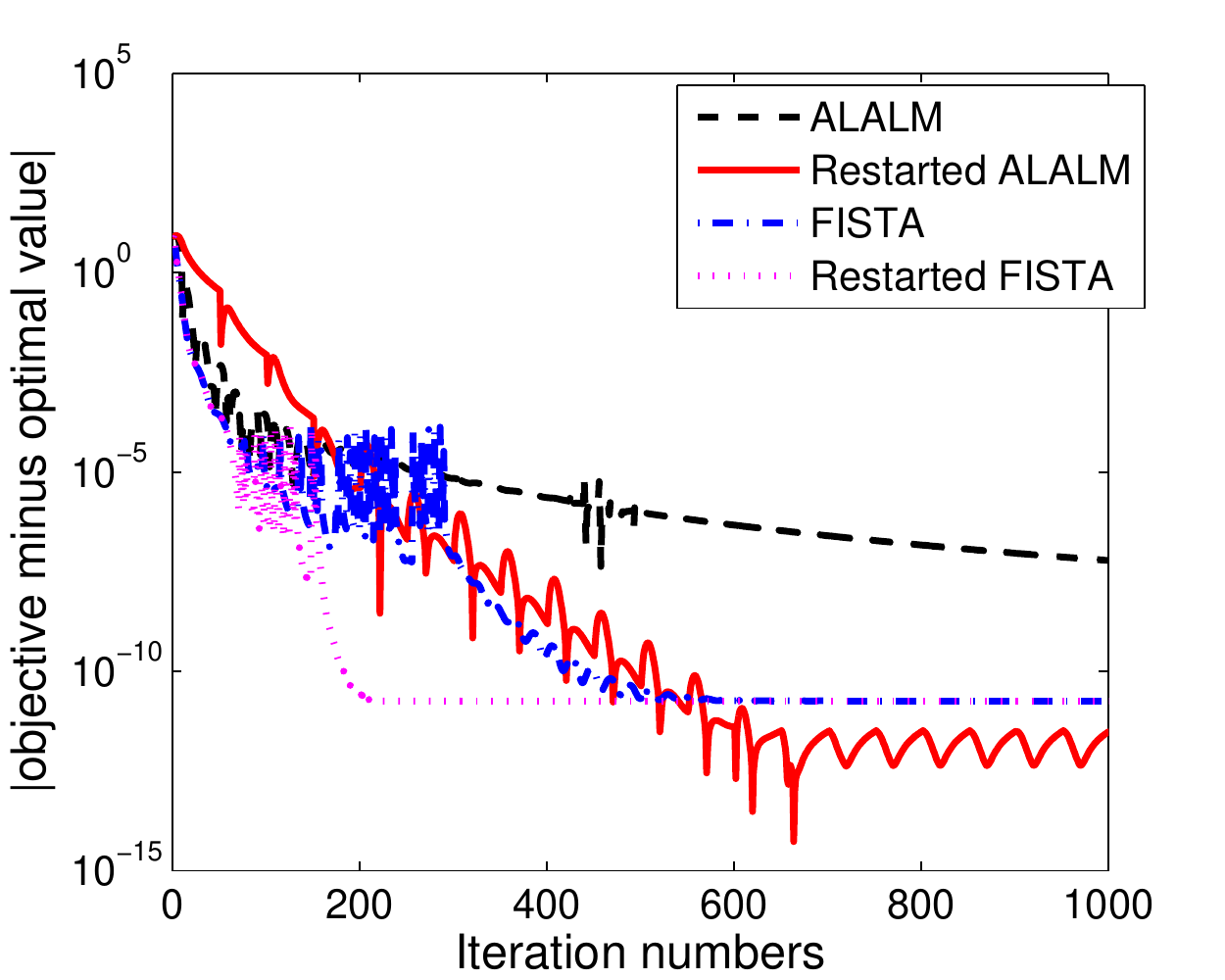}\\
\includegraphics[width=0.25\textwidth]{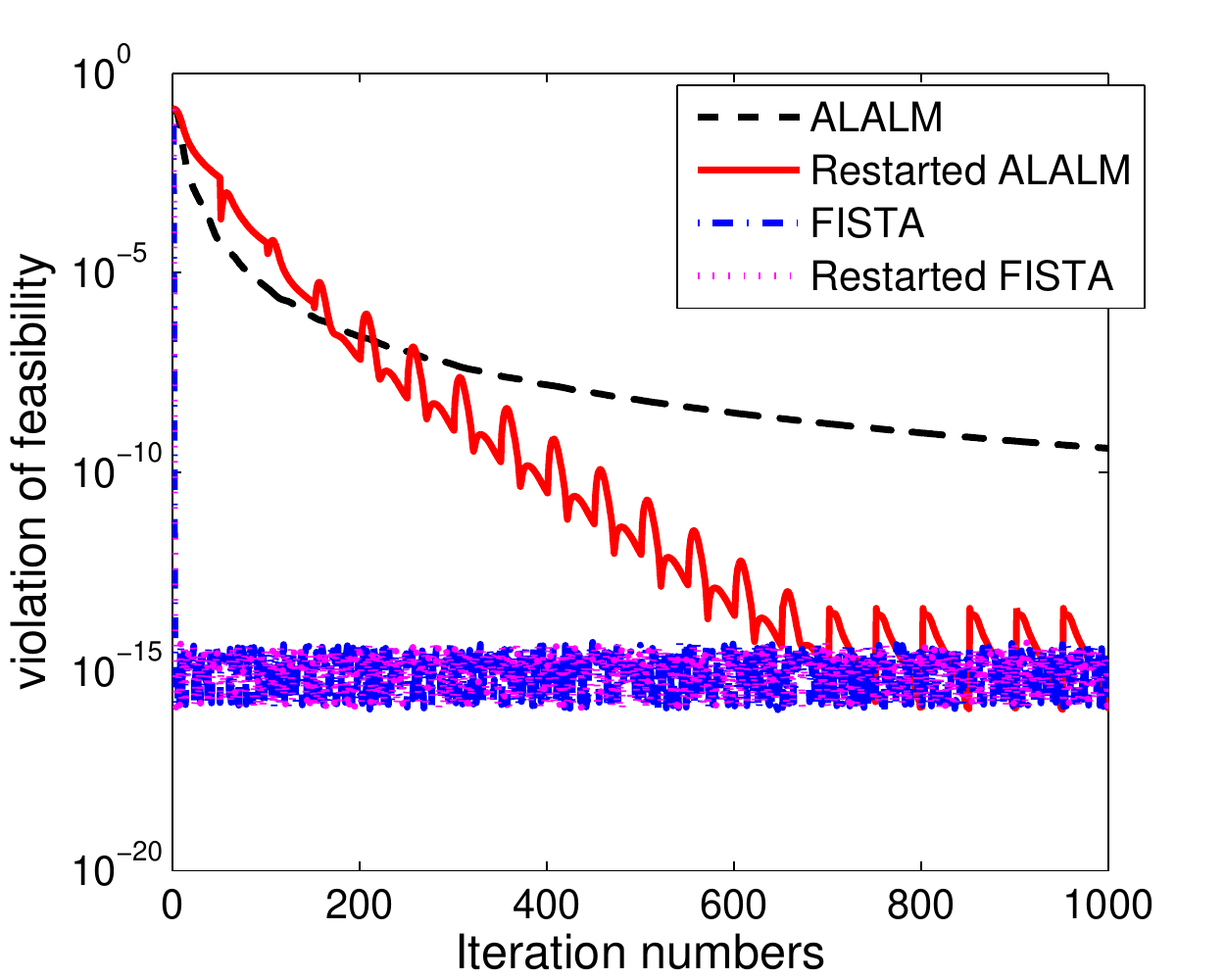} &
\includegraphics[width=0.25\textwidth]{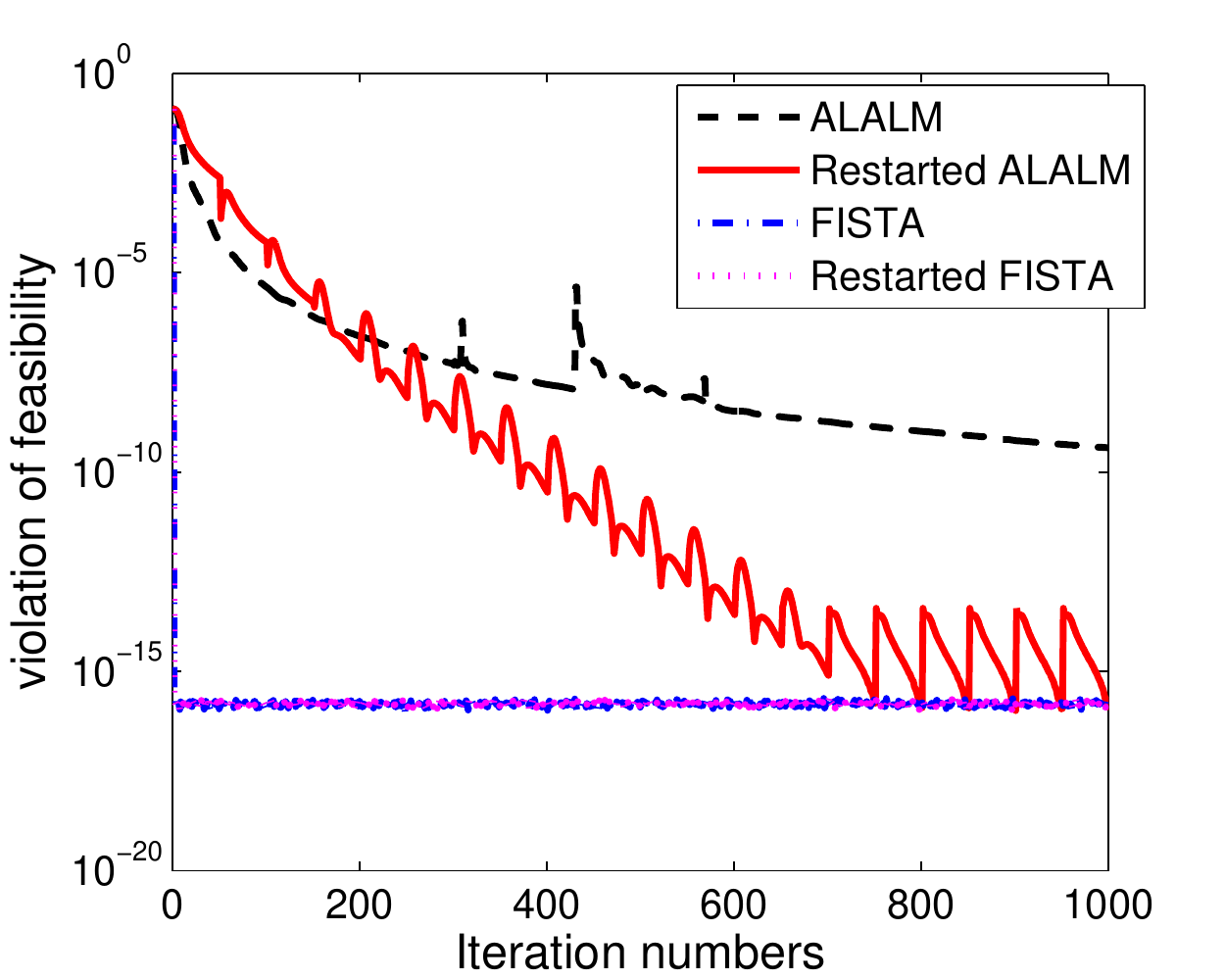} &
\includegraphics[width=0.25\textwidth]{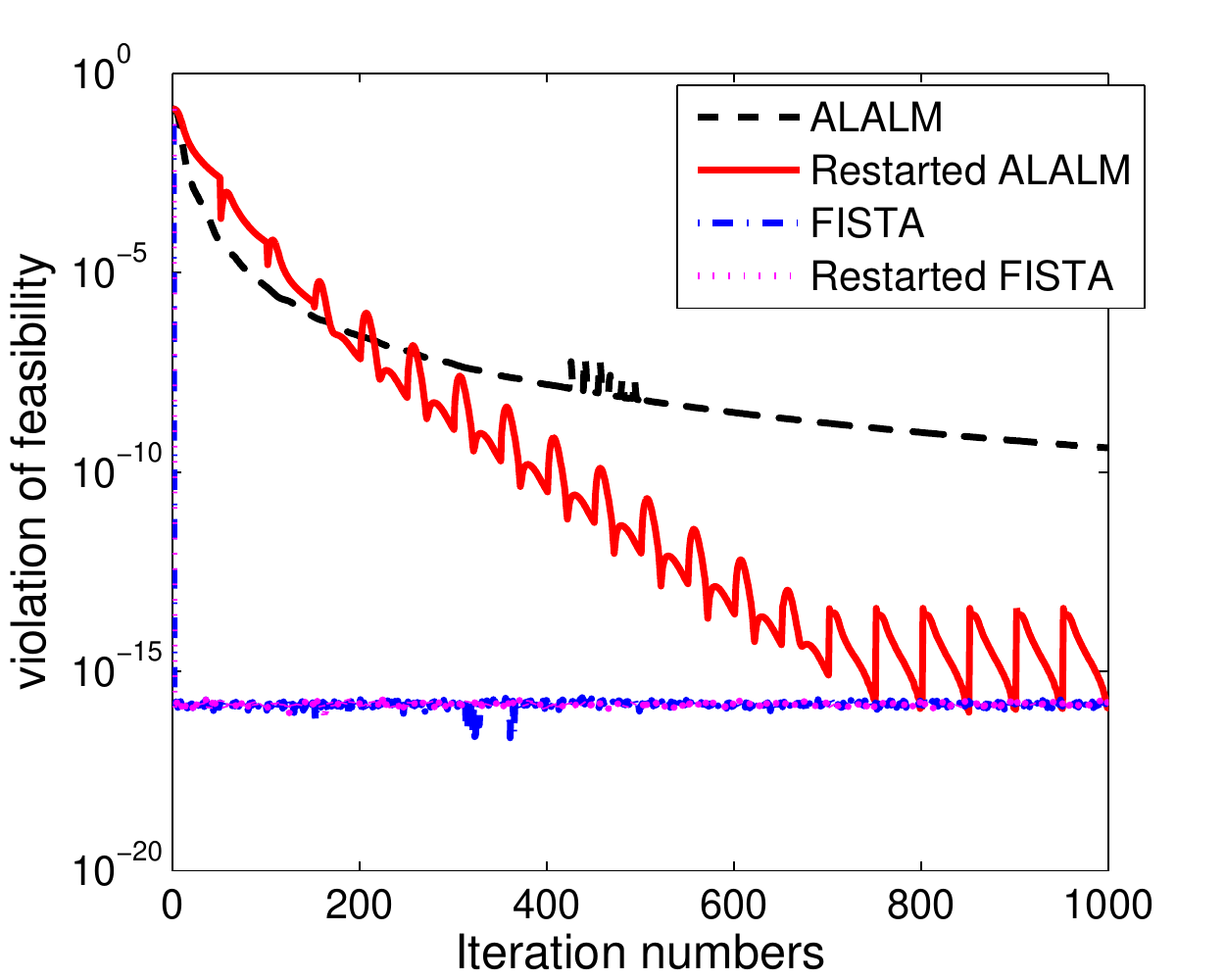}
\end{tabular}
\end{center}
\caption{Results by FISTA \cite{FISTA2009} and ALALM (Algorithm \ref{alg:alalm} with adaptive parameters) on solving \eqref{eq:quadprog} where $A=[B,I]$ and $B$ is generated according to standard Gaussian distribution. Subproblems for both methods are solved to a tolerance specified by \protect\UseVerb{subtol}. First row: the absolute value of objective value minus the optimal value $|F(x)-F(x^*)|$; second row: the violation of feasibility $\|Ax-b\|$.}\label{fig:qp-G50}
\end{figure}


\begin{figure}
\begin{center}
\begin{tabular}{ccc}
$\verb|subtol|=10^{-6}$ & $\verb|subtol|=10^{-8}$ & $\verb|subtol|=10^{-10}$\\
\includegraphics[width=0.25\textwidth]{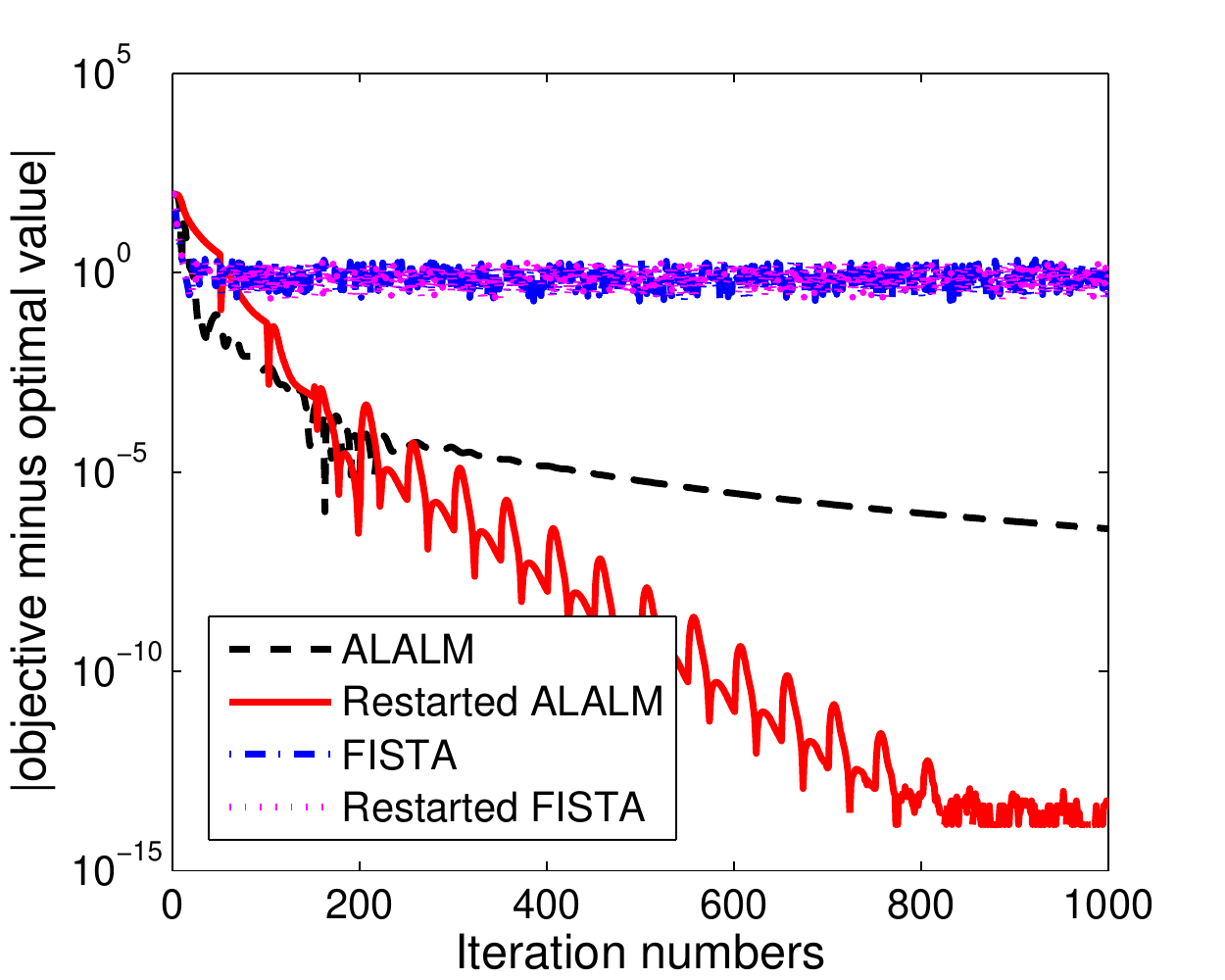} &
\includegraphics[width=0.25\textwidth]{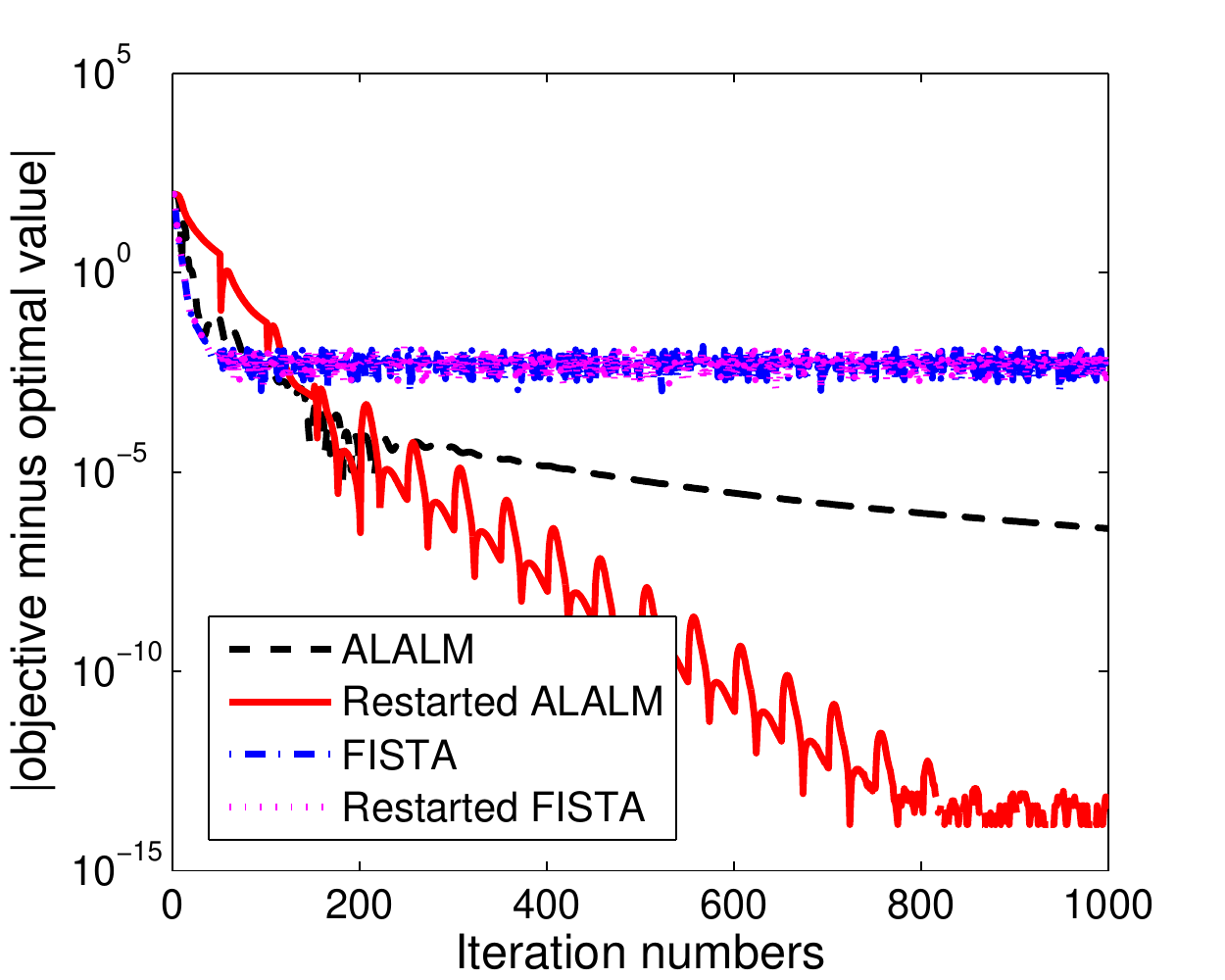} &
\includegraphics[width=0.25\textwidth]{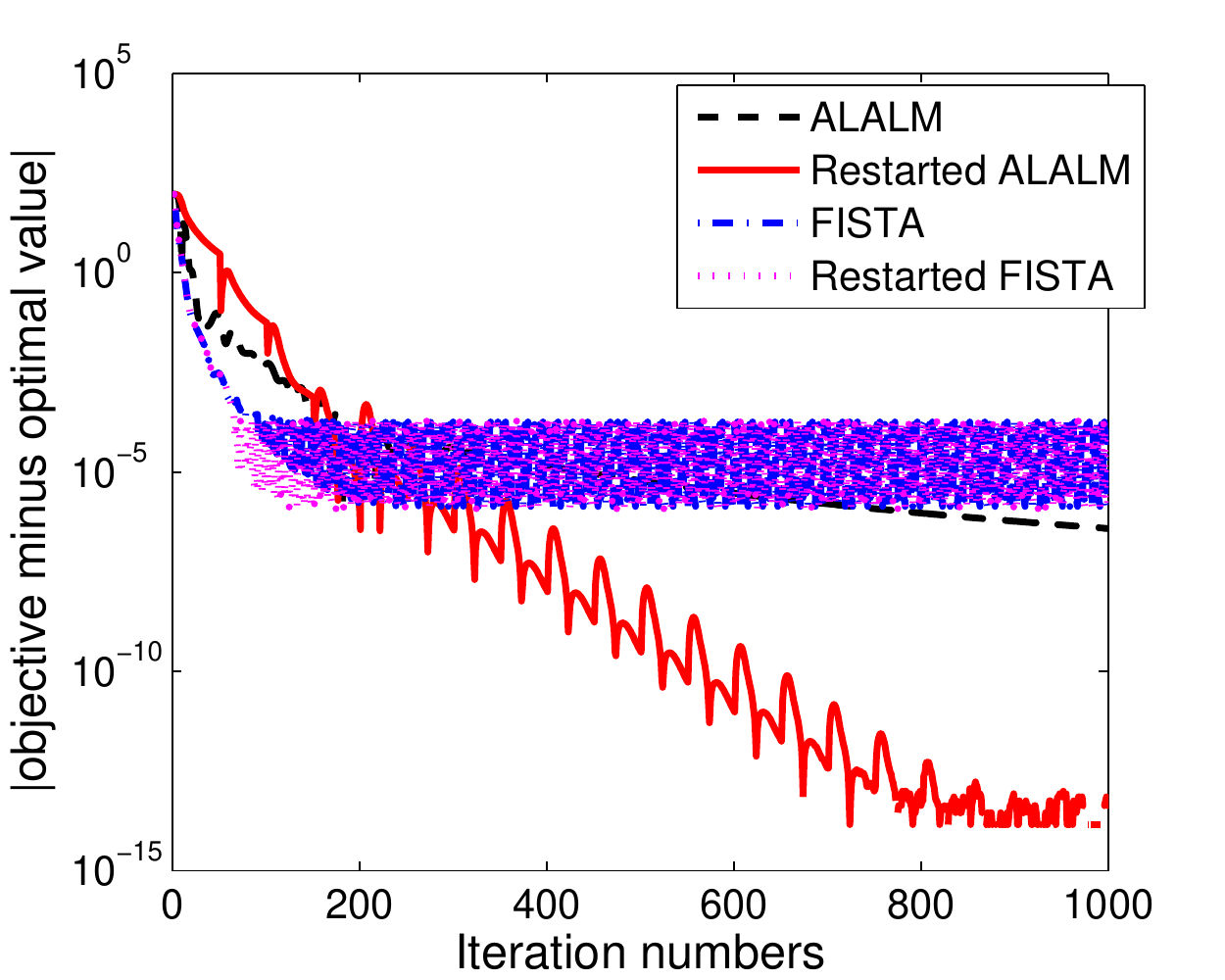}\\
\includegraphics[width=0.25\textwidth]{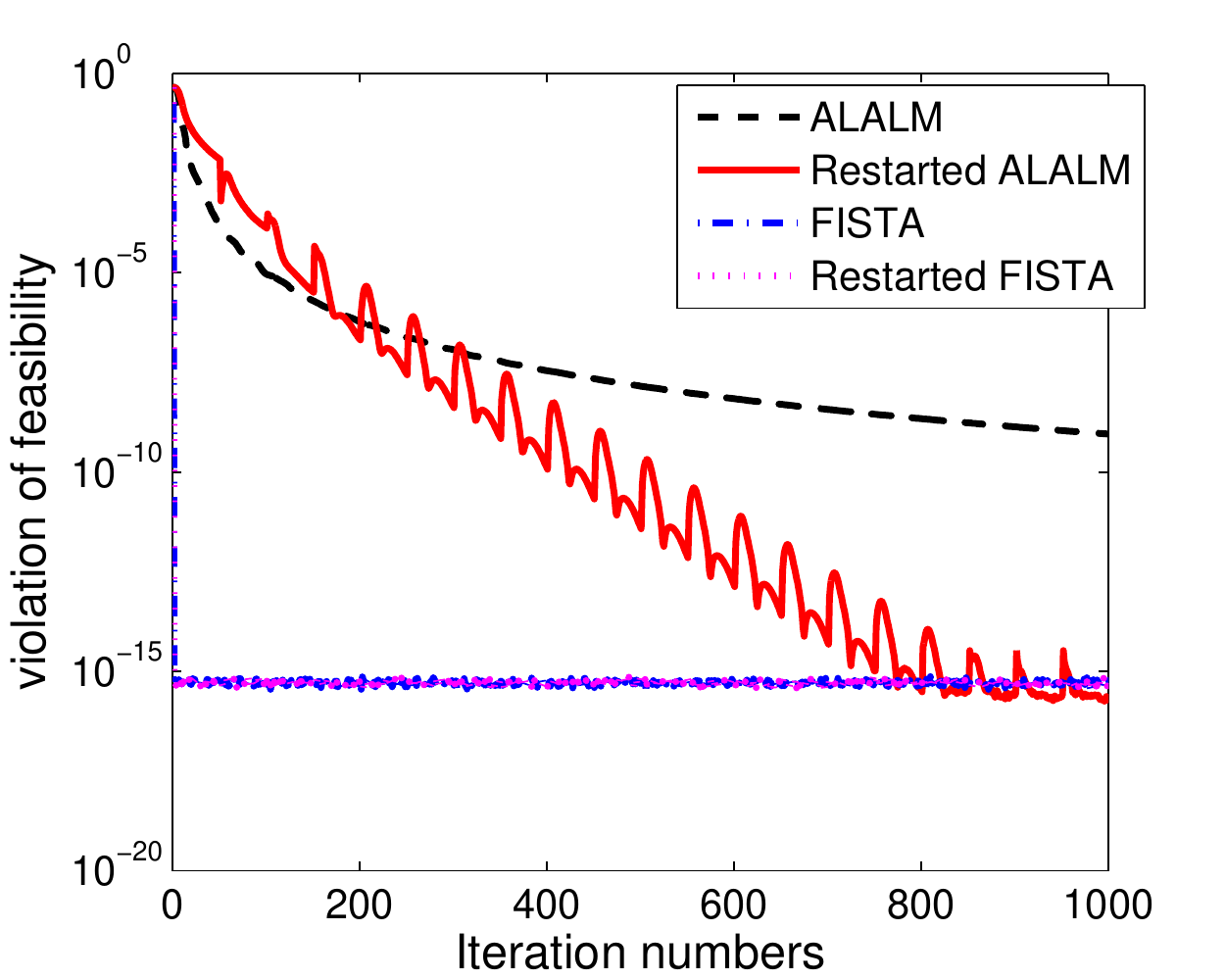} &
\includegraphics[width=0.25\textwidth]{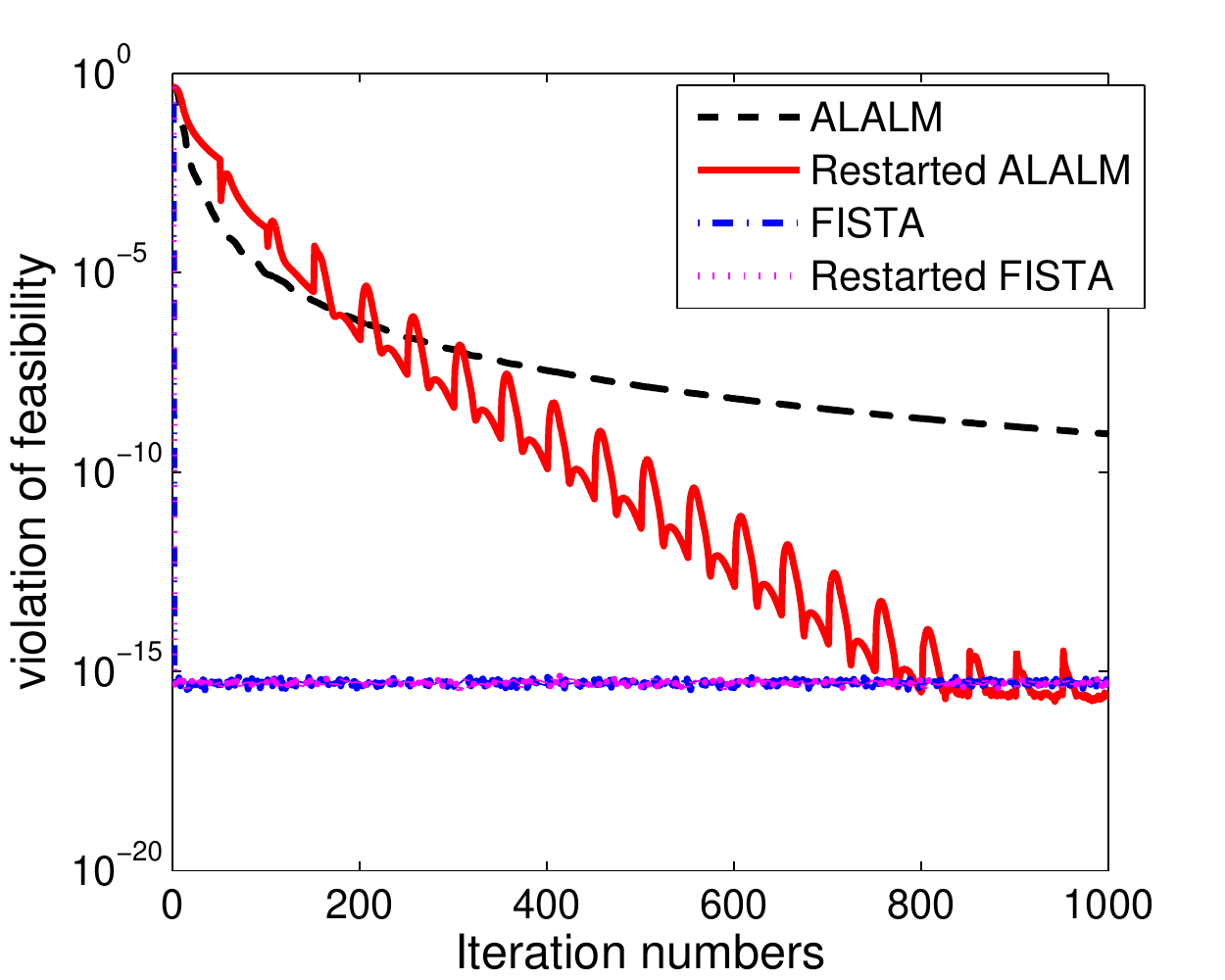} &
\includegraphics[width=0.25\textwidth]{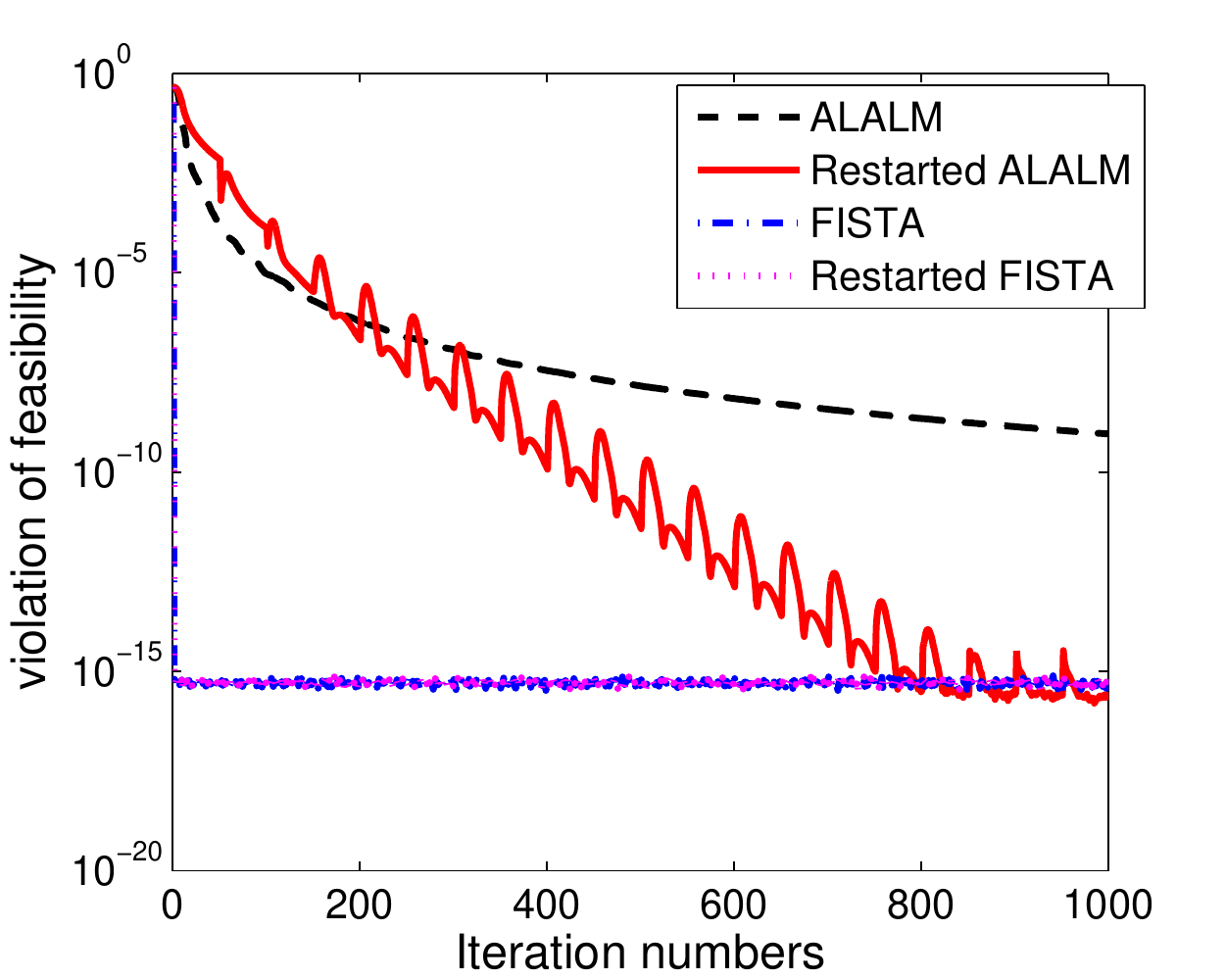}
\end{tabular}
\end{center}
\caption{Results by FISTA \cite{FISTA2009} and ALALM (Algorithm \ref{alg:alalm} with adaptive parameters) on solving \eqref{eq:quadprog} where $A=[B,I]$ and $B$ is generated according to uniform distribution. Subproblems for both methods are solved to a tolerance specified by \protect\UseVerb{subtol}. First row: the absolute value of objective value minus the optimal value $|F(x)-F(x^*)|$; second row: the violation of feasibility $\|Ax-b\|$.}\label{fig:qp-U50}
\end{figure}


\subsection{Image denoising}\label{sec:denois}
In this subsection, we test the accelerated ADMM, i.e., Algorithm \ref{alg:apadmm}, on the total variation regularized image denoising problem:
\begin{equation}\label{eq:denois}
\min_X F(X)=\frac{1}{2}\|X-M\|_F^2+\mu\|\cD X\|_1,
\end{equation}
where $M$ is a noisy two-dimensional image, $\cD$ is a finite difference operator, and $\|Y\|_1=\sum_{i,j}|Y_{ij}|$. Replacing $\cD X$ by $Y$, we can write \eqref{eq:denois} equivalently to 
\begin{equation}\label{eq:denois2}
\min_{X,Y} G(X,Y)=\frac{1}{2}\|X-M\|_F^2+\mu\|Y\|_1, \st \cD X=Y.
\end{equation}

Applying Algorithm \ref{alg:apadmm} to \eqref{eq:denois2} gives the updates:
\begin{subequations}\label{eq:admm-denois}
\begin{align}
&Y^{k+1}=\argmin_Y \mu\|Y\|_1 + \langle\Lambda^k, Y\rangle+\frac{\beta_k}{2}\|Y-\cD X\|_F^2+\frac{1}{2}\|Y-Y^k\|_{P^k}^2,\label{eq:admm-denois-Y}\\
&X^{k+1}=\argmin_X \frac{1}{2}\|X-M\|_F^2 -\langle\Lambda^k, \cD X \rangle+\frac{\beta_k}{2}\|Y-\cD X\|_F^2+\frac{1}{2}\|X-X^k\|_{Q^k}^2,\label{eq:admm-denois-X}\\
&\Lambda^{k+1}=\Lambda^k-\gamma_k(\cD X^{k+1}-Y^{k+1}).
\end{align}
\end{subequations}
We test the algorithm with four sets of parameters, leading to four different methods listed below: 
\begin{itemize}
\item\text{Nonaccelerated ADMM:} $\beta_k=\gamma_k=10,\, P^k=0,\, Q^k=0,\,\forall k$;
\item\text{Accelerated ADMM:} $\beta_k=\gamma_k=\frac{k+1}{2\|\cD\|_2^2},\, P^k=0,\, Q^k=0,\,\forall k$;
\item\text{Nonacclerated Linearized ADMM:} $\beta_k=\gamma_k=\frac{1}{2\|\cD\|_2^2},\, P^k=0,\, Q^k=\frac{I}{2}-\frac{\cD^\top\cD}{2\|\cD\|_2^2},\,\forall k$;
\item\text{Accelerated Linearized ADMM:} $\beta_k=\gamma_k=\frac{k+1}{20\|\cD\|_2^2},\, P^k=0,\, Q^k=\frac{(k+1)I}{20}-\frac{(k+1)\cD^\top\cD}{20\|\cD\|_2^2},\,\forall k$.
\end{itemize}
With $P^k=0$,  the solution of \eqref{eq:admm-denois-Y} can be written analyticly by using the soft thresholding or shrinkage. We assume periodic boundary condition, and thus with $Q^k=0$, the solution of \eqref{eq:admm-denois-X} can be easily obtained by solving a linear system that involves one two-dimensional fast Fourier transform (FFT2) and one inverse FFT2 and some componentwise division \cite{wang2008new}. For the linearized ADMM, it is easy to write closed form solutions for both $X$ and $Y$ subproblems. We compare Algorithm \ref{alg:apadmm} with the above four settings to the accelerated primal-dual method in \cite{chambolle2011first}, which we call Chambolle-Pock method by authors' name. As shown in \cite{GXZ-RPDCU2016}, Chambolle-Pock method is equivalent to linearized ADMM applied to the dual reformulation of \eqref{eq:denois}. It iteratively performs the updates:
\begin{subequations}
\begin{align}
&Z^{k+1}=\argmin_{|Z_{ij}|\le 1,\forall i,j} \|Z-Z^k-\sigma_k \cD \bar{X}^k\|_F^2,\\
&X^{k+1}=\argmin_X \frac{\tau_k}{2\mu}\|X-X^k\|_F^2+\frac{1}{2}\|X-X^k+\tau_k\cD^* Z^{k+1}\|_F^2,\\
&\bar{X}^{k+1}=X^{k+1}+\theta_k(X^{k+1}-X^k)
\end{align}
\end{subequations}
with $\bar{X}^1=X^1$, $\tau_1\sigma_1\|\cD\|_2^2\le 1$, and the parameters set to 
$$\theta_k=\frac{1}{\sqrt{1+2\gamma\tau_k}},\, \tau_{k+1}=\theta_k\tau_k,\, \sigma_{k+1}=\frac{\sigma_k}{\theta_k}\,\forall k.$$
We set $\tau_1=\sigma_1=1/\|\cD\|_2$ and $\gamma=0.35/\mu$ as suggested in \cite{chambolle2011first}. 

In this test, we use the Cameraman image shown in Figure \ref{fig:cameraman}, and we add 10\% Gaussian noise. The regularization parameter is set to $\mu=0.04$. For Algorithm \ref{alg:apadmm}, we report the objective value of \eqref{eq:denois2} and the violation of feasibility and also the objective value of \eqref{eq:denois}, and for Chambolle-Pock method we only report the objective value of \eqref{eq:denois} since it solves the dual problem and does not guarantee the feasibility of \eqref{eq:denois2}. Figure \ref{fig:denois} plots the results in terms of iteration numbers. Since the linearized ADMM and Chambolle-Pock methods has lower iteration complexity than the nonlinearized ADMM, we also plot the results in terms of running time. From the figure, we see that Algorithm \ref{alg:apadmm} with adaptive parameters performs significantly better than that with fixed parameters. The Chambolle-Pock method decreases the objective fastest in the beginning, and later the accelerated ADMM with or without linearization catch up and surpass it. 

\begin{figure}
\begin{center}
\begin{tabular}{ccc}
{\small original image} & {\small noisy image} & {\small denoised image}\\
\includegraphics[width=0.25\textwidth]{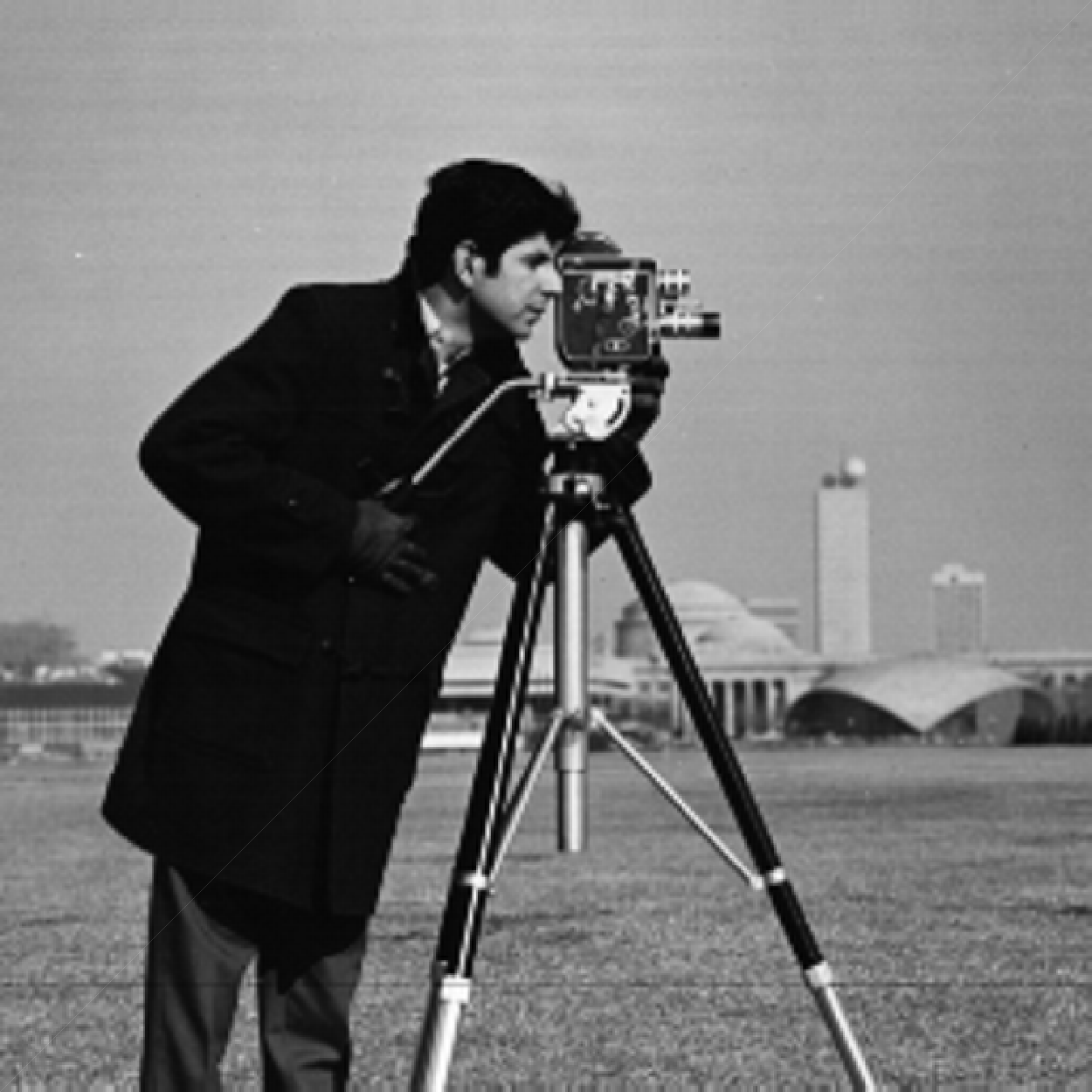} &
\includegraphics[width=0.25\textwidth]{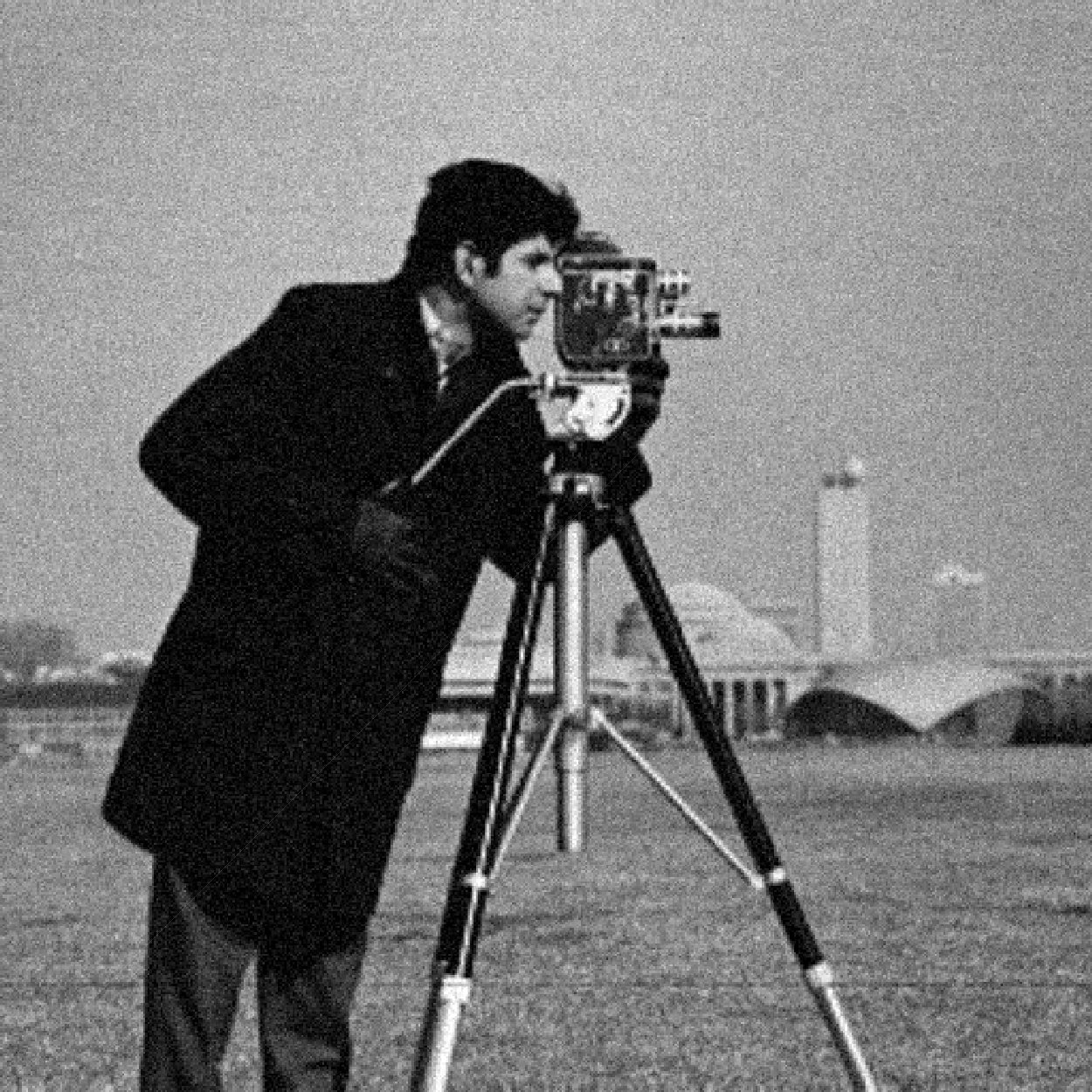} &
\includegraphics[width=0.25\textwidth]{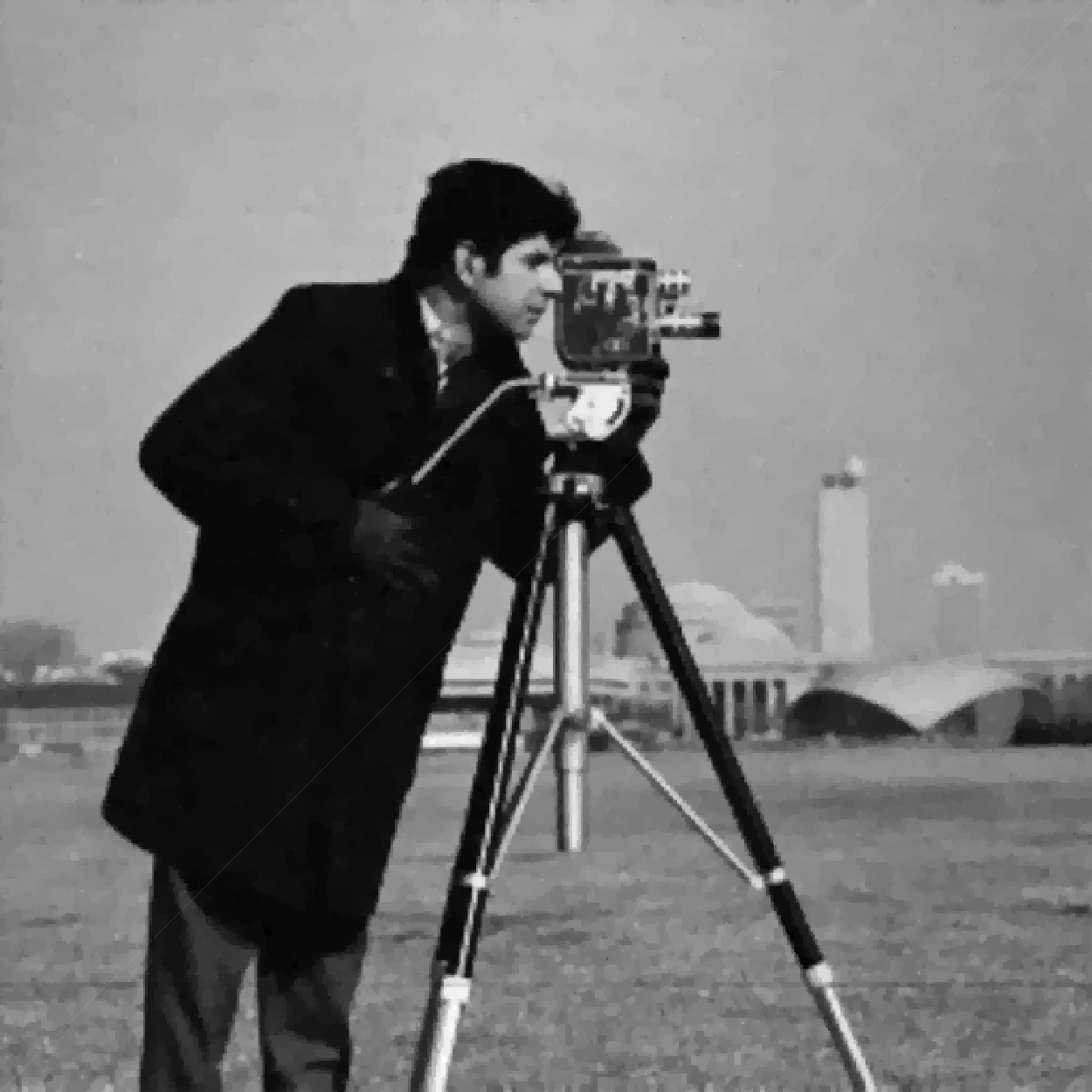}
\end{tabular}
\end{center}
\caption{The Cameraman images. Left: original one; Middle: noisy image with 10\% Gaussian noise, PSNR = 25.62; Right: denoised image by the accelerated ADMM running to 200 iterations, PSNR = 33.29.}\label{fig:cameraman}
\end{figure}

\begin{figure}
\begin{center}
\begin{tabular}{cc}
{\small objective measure of \eqref{eq:denois2}} & {\small feasibility measure of \eqref{eq:denois2}} \\
\includegraphics[width=0.25\textwidth]{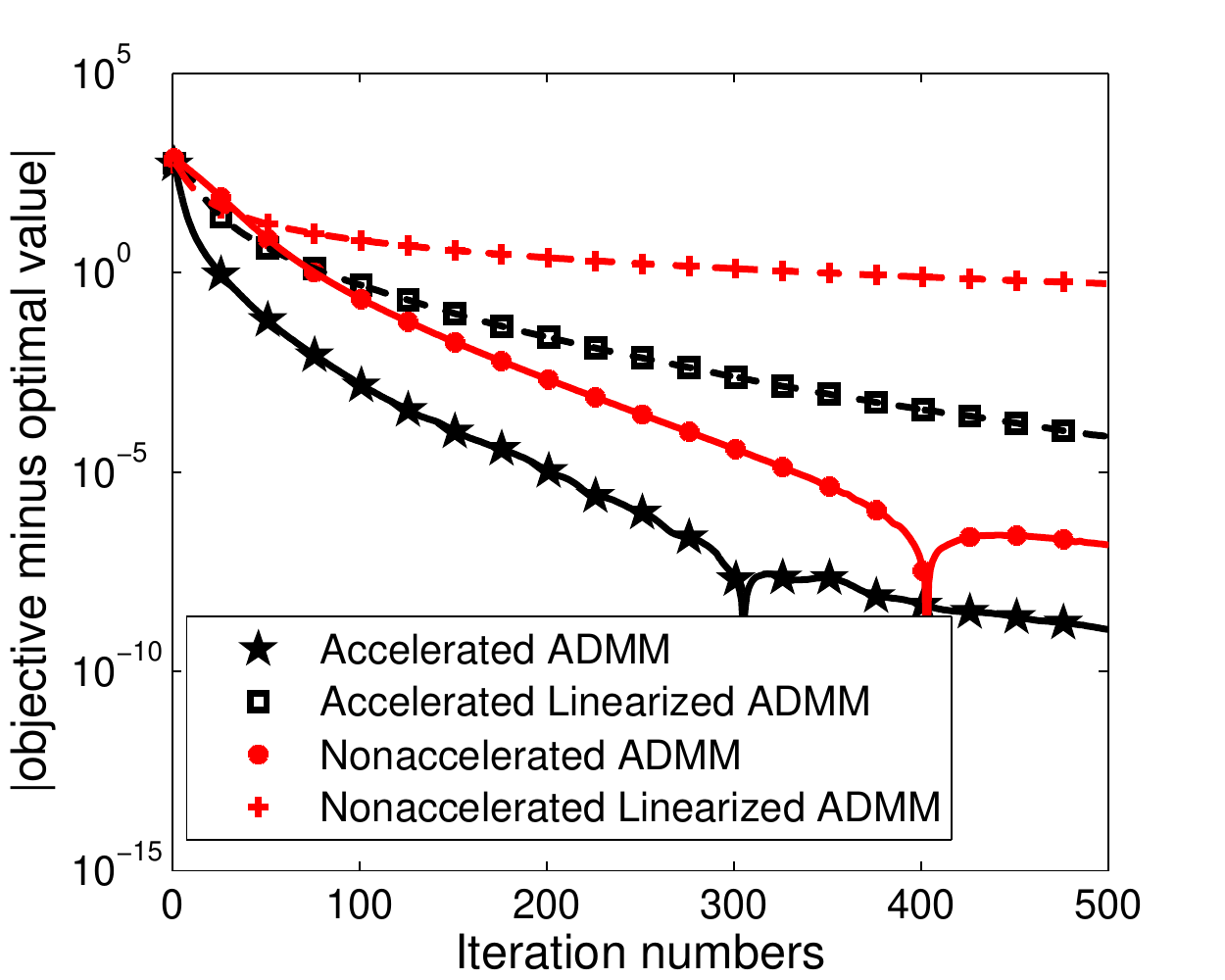} &
\includegraphics[width=0.25\textwidth]{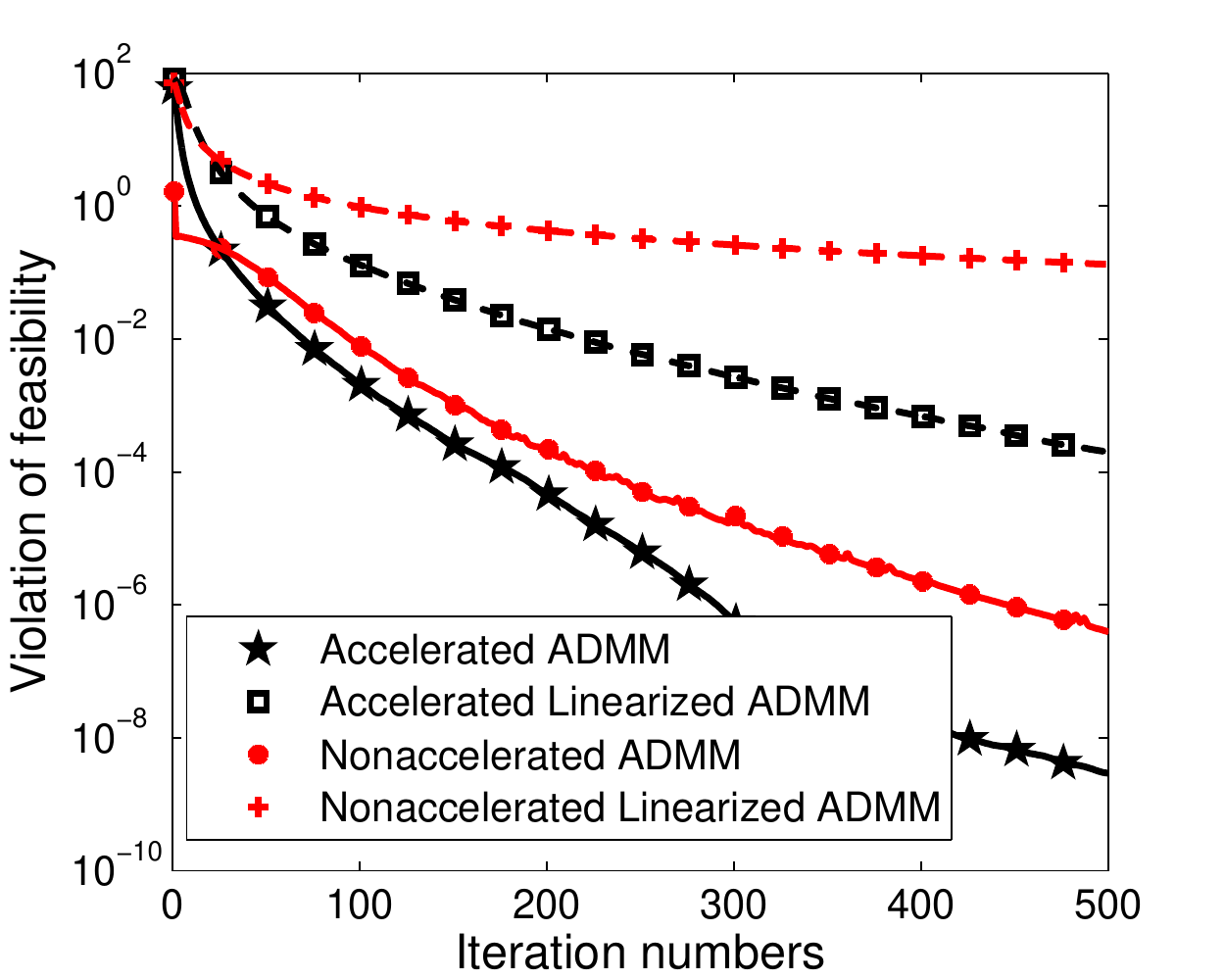}\\
{\small objective measure of \eqref{eq:denois}} & {\small objective measure of \eqref{eq:denois}}\\
\includegraphics[width=0.25\textwidth]{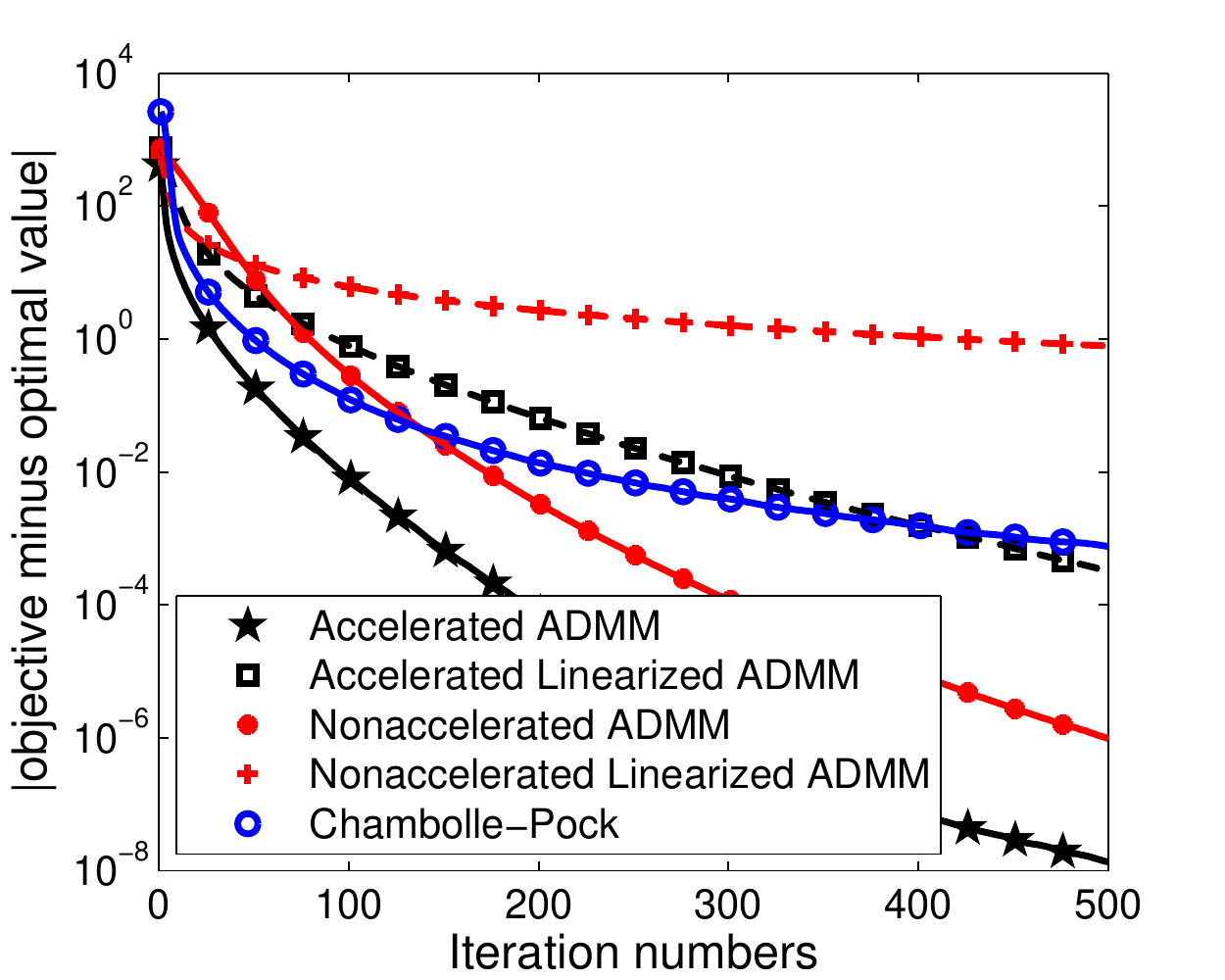}&
\includegraphics[width=0.25\textwidth]{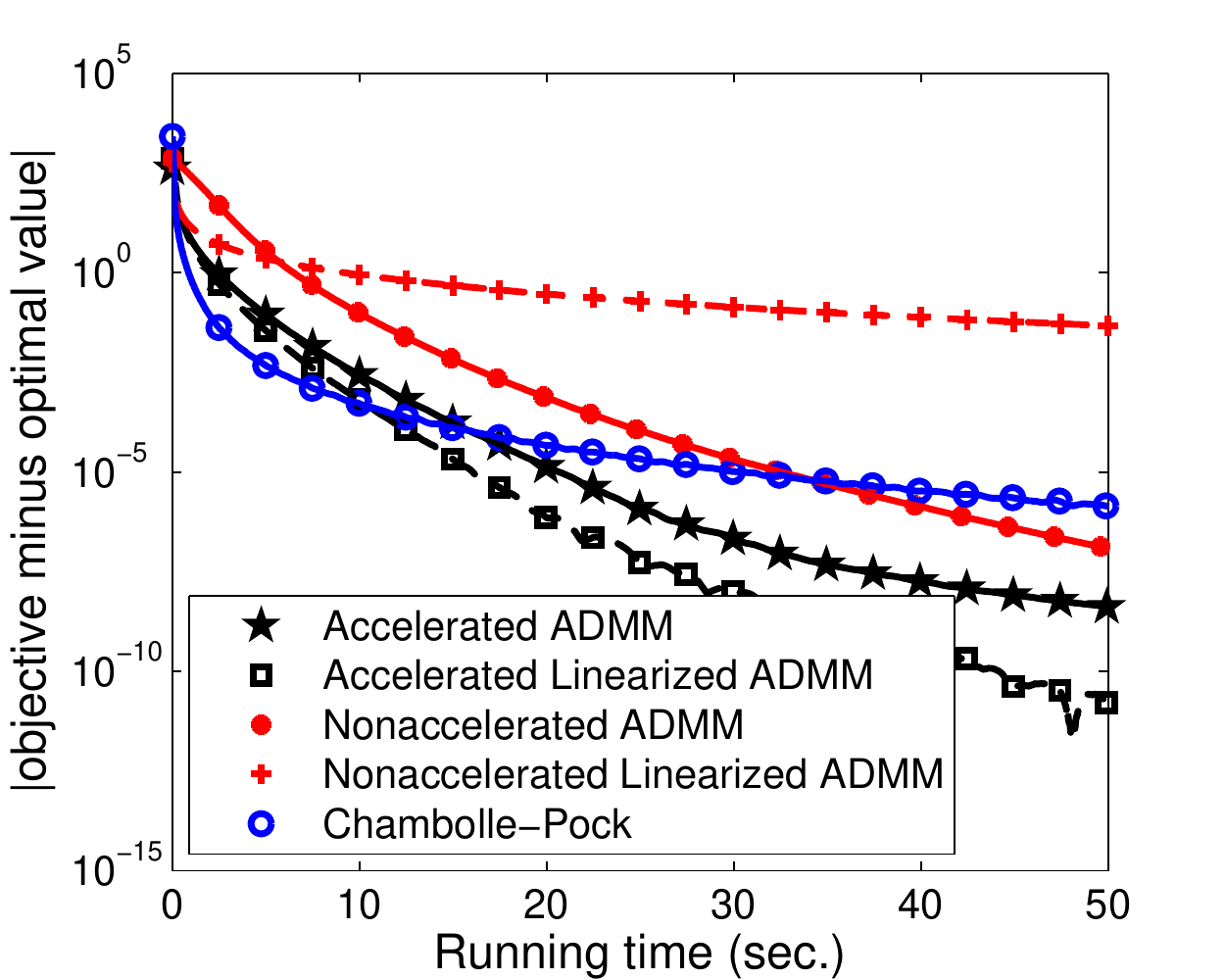}
\end{tabular}
\end{center}
\caption{Results by Algorithm \ref{alg:apadmm} with adaptive parameters (accelerated ADMM) and constant parameters (nonaccelerated ADMM) and also the Chambolle-Pock method on solving \eqref{eq:denois}. Top left: the absolute value of objective of \eqref{eq:denois2} minus optimal value $|G(X,Y)-G(X^*,Y^*)|$; Top right: the violation of feasibility of \eqref{eq:denois2} $\|\cD X-Y\|_F$; Bottom left: the absolute value of objective of \eqref{eq:denois} minus optimal value $|F(X)-F(X^*)|$ in terms of iteration; Bottom right: the absolute value of objective of \eqref{eq:denois} minus optimal value $|F(X)-F(X^*)|$ in terms of running time.}\label{fig:denois}
\end{figure}

\subsection{Elastic net regularized support vector machine}
We test Algorithm \ref{alg:apadmm} on the elastic net regularized support vector machine problem
\begin{equation}\label{eq:ensvm}
\min_x F(x)=\frac{1}{m}\sum_{i=1}^m [1-b_ia_i^\top x]_+ +\mu_1\|x\|_1+\frac{\mu_2}{2}\|x\|^2,
\end{equation}
where $[c]_+=\max(0,c)$, $\{(a_i,b_i)\}_{i=1}^m$ are the samples in $p$-dimensional space, and $b_i\in\{+1,-1\}$ is the label of the $i$th sample. Let $A=[a_1,\ldots,a_m]\in\RR^{p\times m}$ and replace $1-b_ia_i^\top x$ by $y_i$ for all $i$. We obtain the equivalent formulation:
\begin{equation}\label{eq:ensvm2}
\min_x G(x,y)=\frac{1}{m}e^\top [y]_+ +\mu_1\|x\|_1+\frac{\mu_2}{2}\|x\|^2, \st Bx+y=e,
\end{equation}
where $e$ is the vector with all ones, and $B=\text{Diag}(b)A$.

The data is generated in the same way as that in \cite{xu2015HHSVM}. One half of the samples belong to ``+1'' class and the other to ``-1'' class. Each sample in ``+1'' class is generated according to Gaussian distribution $\cN(u,\Sigma)$, and each sample in ``-1'' class follows  $\cN(-u,\Sigma)$. The mean vector and variance matrix are set to
$$u=\left[\begin{array}{l}E_{s\times 1}\\0_{(p-s)\times 1}\end{array}\right],\quad \Sigma=\left[\begin{array}{cc}\rho E_{s\times s}+\rho I_{s\times s} & 0_{s\times (p-s)}\\ 0_{(p-s)\times s} & I_{(p-s)\times (p-s)}\end{array}\right],$$
where $E_{s\times s}$ is an $s\times s$ matrix with all ones, $s$ is the number of features that are related to classification, and $\rho\in [0,1]$ measures the correlation of the features (the larger it is, the harder the problem is). In the test, we set
$m = 100,\, p = 500,\, s = 50, \rho = 0.5$ and $\mu_1=\mu_2=0.01$.

Applying Algorithm \ref{alg:apadmm} to \eqref{eq:ensvm2}, we iteratively perform the updates:
\begin{subequations}\label{eq:admm-ensvm}
\begin{align}
&y^{k+1}=\argmin_y \frac{1}{m}e^\top [y]_+ -\langle \lambda^k, y\rangle +\frac{\beta_k}{2}\|Bx^k+y-e\|^2+\frac{1}{2}\|y-y^k\|_{P^k}^2,\\
&x^{k+1}=\argmin_x \mu_1\|x\|_1+\frac{\mu_2}{2}\|x\|^2-\langle\lambda^k, Bx\rangle+\frac{\beta_k}{2}\|Bx+y^{k+1}-e\|^2+\frac{1}{2}\|x-x^k\|_{Q^k}^2,\label{eq:admm-ensvm-x}\\
&\lambda^{k+1}=\lambda^k-\gamma_k(Bx^{k+1}+y^{k+1}-e).
\end{align}
\end{subequations}
Again, we test two sets of parameters. The first one fixes the parameters during all iterations, and the second one adapts the parameters. Since the coexistence of $\ell_1$-norm and the least squares term makes \eqref{eq:admm-ensvm-x} difficult to solve, we choose $Q^k$ to cancel the term $x^\top B^\top Bx$, i.e., we linearize the augmented term. Specifically, we set the parameters in the same way as the previous test:
\begin{itemize}
\item Nonaccelerated Linearized ADMM: $\beta_k=\gamma_k=\frac{1}{2\|B\|_2^2},\, P^k=0,\, Q^k=\frac{I}{2}-\frac{B^\top B}{2\|B\|_2^2},\,\forall k$;
\item Accelerated Linearized ADMM: $\beta_k=\gamma_k=\frac{\mu_2(k+1)}{20\|B\|_2^2},\, P^k=0,\, Q^k=\frac{\mu_2(k+1)I}{20}-\frac{\mu_2(k+1)B^\top B}{20\|B\|_2^2},\,\forall k$.
\end{itemize}
We also compare the linearized ADMM to the classic ADMM without linearization, which introduces another variable $z$ to split $x$ from the $\ell_1$-norm and solves the problem
\begin{equation}\label{eq:ensvm3}
\min_x \frac{1}{m}e^\top [y]_+ +\mu_1\|z\|_1+\frac{\mu_2}{2}\|x\|^2, \st Bx+y=e,\, x=z.
\end{equation}
We use the code from \cite{ye2011efficient} to solve \eqref{eq:ensvm3} and tune its parameters as best as we can.

Similar to the previous test, we measure the objective value and feasibility of \eqref{eq:ensvm2} given by the linearized ADMM and the objective value of \eqref{eq:ensvm} for all three methods. Figure \ref{fig:svm} plots the results, from which we see that the accelerated linearized ADMM performs significantly better than the nonaccelerated counterpart, and the latter is comparable to the classic nonlinearized ADMM.

\begin{figure}
\begin{center}
\begin{tabular}{ccc}
{\small objective measure of \eqref{eq:ensvm2}} & {\small feasibility measure of \eqref{eq:ensvm2}} & {\small objective measure of \eqref{eq:ensvm}}\\
\includegraphics[width=0.25\textwidth]{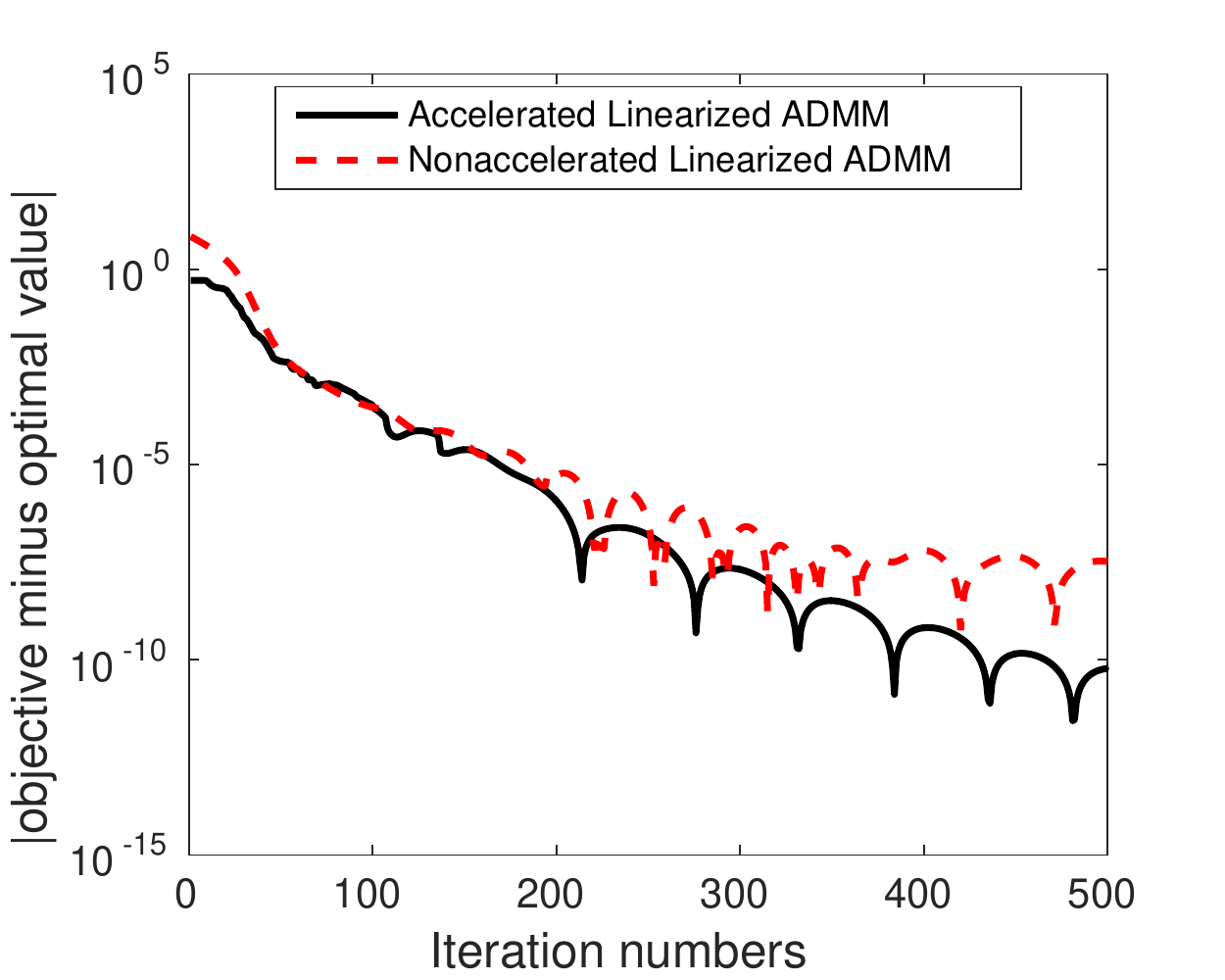} &
\includegraphics[width=0.25\textwidth]{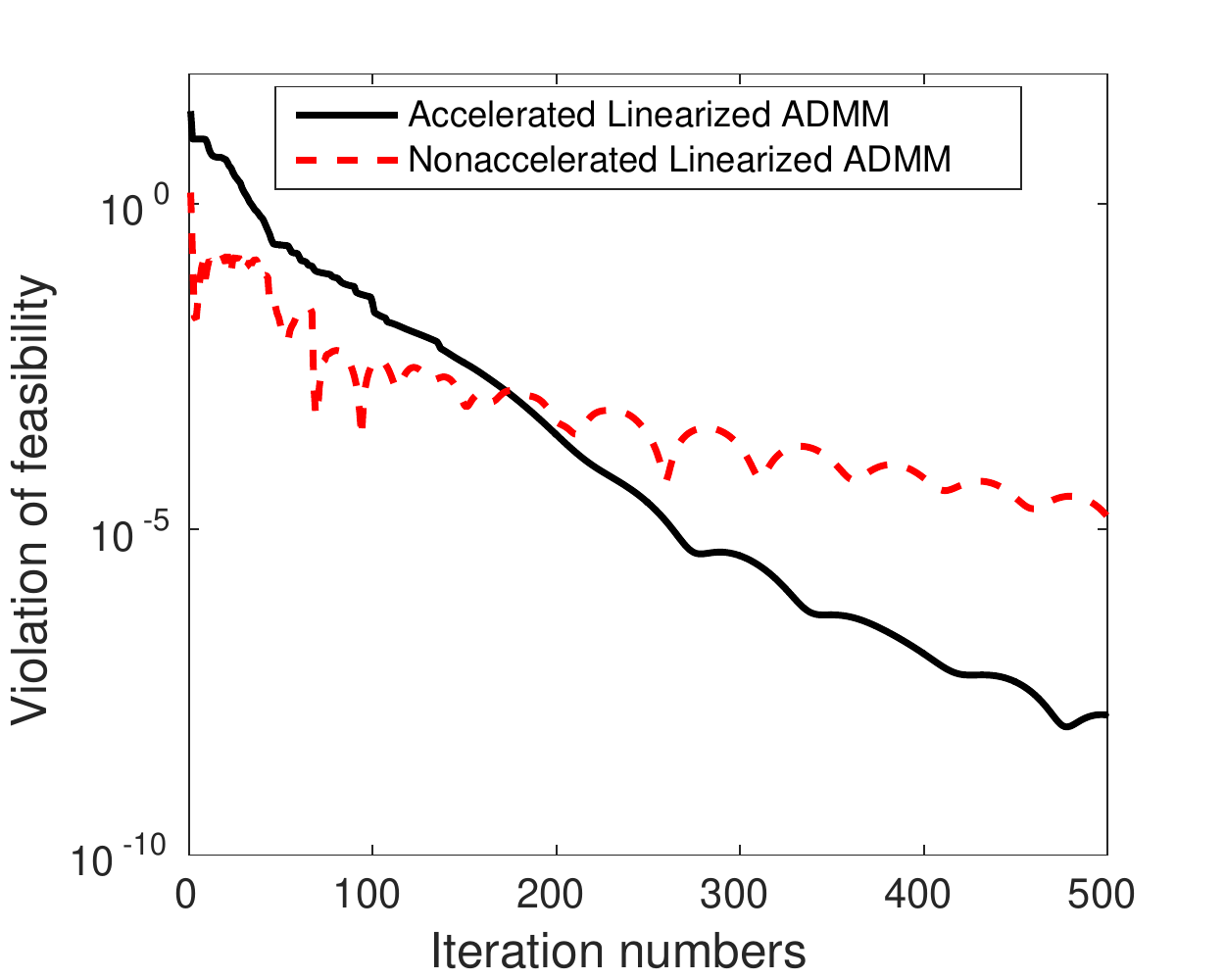} &
\includegraphics[width=0.25\textwidth]{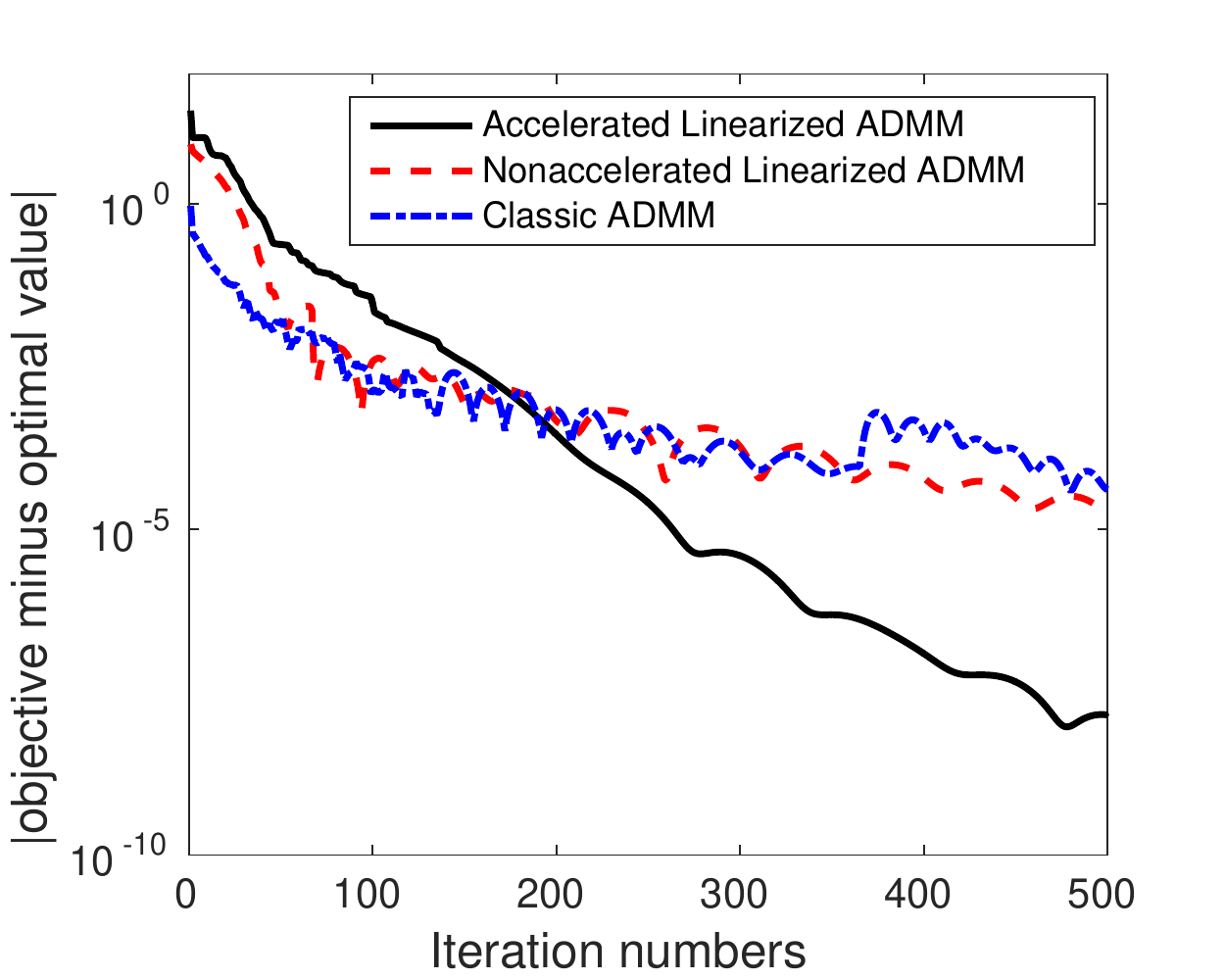}
\end{tabular}
\end{center}
\caption{Results by Algorithm \ref{alg:apadmm} with adaptive parameters (accelerated linearized ADMM) and constant parameters (nonaccelerated linearized ADMM) and also the classic nonlinearized ADMM on solving \eqref{eq:ensvm}. Left: the absolute value of objective of \eqref{eq:ensvm2} minus optimal value $|G(x,y)-G(x^*,y^*)|$; Middle: the violation of feasibility of \eqref{eq:ensvm2} $\|Bx+y-e\|$; Right: the absolute value of objective of \eqref{eq:ensvm} minus optimal value $|F(x)-F(x^*)|$.}\label{fig:svm}
\end{figure}

\section{Conclusions}\label{sec:conclusion}
We have proposed an accelerated linearized augmented Lagrangian method (ALALM) and also an accelerated alternating direction method of multipliers (ALADMM) for solving structured linearly constrained convex programming. We have established $O(1/t^2)$ convergence rate for ALALM by assuming merely weak convexity and for ALADMM by assuming strong convexity to one block variable. Numerical experiments have been performed to demonstrate the validness of acceleration and higher efficiency over existing accelerated methods.

To have the $O(1/t^2)$ convergence rate for the ALALM, our current analysis does not allow linearization to the augmented term, and that may cause great difficulty on solving subproblems if meanwhile we have a complicated nonsmooth term. It is interesting to know whether we can linearize the augmented term and still obtain $O(1/t^2)$ convergence under the same assumptions. We are unable to show this under the setting of Algorithm \ref{alg:alalm}, so it may have to turn to other acceleration technique. We leave this open question to interested readers.

\bibliographystyle{abbrv}
\bibliography{alm}

\begin{thebibliography}{10}

\bibitem{FISTA2009}
A.~Beck and M.~Teboulle.
\newblock A fast iterative shrinkage-thresholding algorithm for linear inverse
  problems.
\newblock {\em SIAM journal on imaging sciences}, 2(1):183--202, 2009.

\bibitem{bertsekas2014constrained}
D.~P. Bertsekas.
\newblock {\em Constrained optimization and Lagrange multiplier methods}.
\newblock Academic press, 2014.

\bibitem{bredies2016accelerated}
K.~Bredies and H.~Sun.
\newblock Accelerated douglas-rachford methods for the solution of
  convex-concave saddle-point problems.
\newblock {\em arXiv preprint arXiv:1604.06282}, 2016.

\bibitem{chambolle2011first}
A.~Chambolle and T.~Pock.
\newblock A first-order primal-dual algorithm for convex problems with
  applications to imaging.
\newblock {\em Journal of Mathematical Imaging and Vision}, 40(1):120--145,
  2011.

\bibitem{chen2014optimal}
Y.~Chen, G.~Lan, and Y.~Ouyang.
\newblock Optimal primal-dual methods for a class of saddle point problems.
\newblock {\em SIAM Journal on Optimization}, 24(4):1779--1814, 2014.

\bibitem{condat2013primal}
L.~Condat.
\newblock A primal--dual splitting method for convex optimization involving
  lipschitzian, proximable and linear composite terms.
\newblock {\em Journal of Optimization Theory and Applications},
  158(2):460--479, 2013.

\bibitem{dang2014randomized}
C.~Dang and G.~Lan.
\newblock Randomized methods for saddle point computation.
\newblock {\em arXiv preprint arXiv:1409.8625}, 2014.

\bibitem{fercoq2015accelerated}
O.~Fercoq and P.~Richt{\'a}rik.
\newblock Accelerated, parallel, and proximal coordinate descent.
\newblock {\em SIAM Journal on Optimization}, 25(4):1997--2023, 2015.

\bibitem{gabay1976dual}
D.~Gabay and B.~Mercier.
\newblock A dual algorithm for the solution of nonlinear variational problems
  via finite element approximation.
\newblock {\em Computers $\&$ Mathematics with Applications}, 2(1):17--40,
  1976.

\bibitem{GXZ-RPDCU2016}
X.~Gao, Y.~Xu, and S.~Zhang.
\newblock Randomized primal-dual proximal block coordinate updates.
\newblock {\em arXiv preprint arXiv:1605.05969}, 2016.

\bibitem{gao2015first}
X.~Gao and S.~Zhang.
\newblock First-order algorithms for convex optimization with nonseparate
  objective and coupled constraints.
\newblock {\em Optimization online}, 3:5, 2015.

\bibitem{ghadimi2016accelerated}
S.~Ghadimi and G.~Lan.
\newblock Accelerated gradient methods for nonconvex nonlinear and stochastic
  programming.
\newblock {\em Mathematical Programming}, 156(1-2):59--99, 2016.

\bibitem{Glowinski1975}
R.~Glowinski and A.~Marrocco.
\newblock Sur l'approximation, par \/{e}l\'{e}ments finis d'ordre un, et la
  r\'{e}solution, par p\'{e}nalisation-dualit\'{e} d'une classe de
  probl\`{e}mes de dirichlet non lin\'{e}aires.
\newblock {\em ESAIM: Mathematical Modelling and Numerical Analysis},
  9(R2):41--76, 1975.

\bibitem{goldstein2014fast}
T.~Goldstein, B.~O'Donoghue, S.~Setzer, and R.~Baraniuk.
\newblock Fast alternating direction optimization methods.
\newblock {\em SIAM Journal on Imaging Sciences}, 7(3):1588--1623, 2014.

\bibitem{he2010aalm}
B.~He and X.~Yuan.
\newblock On the acceleration of augmented lagrangian method for linearly
  constrained optimization.
\newblock {\em Optimization online}, 2010.

\bibitem{he2016accelerated}
Y.~He and R.~D. Monteiro.
\newblock An accelerated hpe-type algorithm for a class of composite
  convex-concave saddle-point problems.
\newblock {\em SIAM Journal on Optimization}, 26(1):29--56, 2016.

\bibitem{huang2013acLBM}
B.~Huang, S.~Ma, and D.~Goldfarb.
\newblock Accelerated linearized bregman method.
\newblock {\em Journal of Scientific Computing}, 54(2-3):428--453, 2013.

\bibitem{kadkhodaie2015accelerated}
M.~Kadkhodaie, K.~Christakopoulou, M.~Sanjabi, and A.~Banerjee.
\newblock Accelerated alternating direction method of multipliers.
\newblock In {\em Proceedings of the 21th ACM SIGKDD International Conference
  on Knowledge Discovery and Data Mining}, pages 497--506. ACM, 2015.

\bibitem{kang2015inexact}
M.~Kang, M.~Kang, and M.~Jung.
\newblock Inexact accelerated augmented lagrangian methods.
\newblock {\em Computational Optimization and Applications}, 62(2):373--404,
  2015.

\bibitem{kang2013accelerated}
M.~Kang, S.~Yun, H.~Woo, and M.~Kang.
\newblock Accelerated bregman method for linearly constrained
  $\ell_1$--$\ell_2$ minimization.
\newblock {\em Journal of Scientific Computing}, 56(3):515--534, 2013.

\bibitem{lan2012optimal}
G.~Lan.
\newblock An optimal method for stochastic composite optimization.
\newblock {\em Mathematical Programming}, 133(1-2):365--397, 2012.

\bibitem{lin2014accelerated}
Q.~Lin, Z.~Lu, and L.~Xiao.
\newblock An accelerated proximal coordinate gradient method.
\newblock In {\em Advances in Neural Information Processing Systems}, pages
  3059--3067, 2014.

\bibitem{nesterov1983method}
Y.~Nesterov.
\newblock A method of solving a convex programming problem with convergence
  rate ${O}(1/k^2)$.
\newblock {\em Soviet Mathematics Doklady}, 27(2):372--376, 1983.

\bibitem{nesterov2013gradient}
Y.~Nesterov.
\newblock Gradient methods for minimizing composite functions.
\newblock {\em Mathematical Programming}, 140(1):125--161, 2013.

\bibitem{nocedal2006numerical}
J.~Nocedal and S.~Wright.
\newblock {\em Numerical optimization}.
\newblock Springer Science \& Business Media, 2006.

\bibitem{ouyang2015accelerated}
Y.~Ouyang, Y.~Chen, G.~Lan, and E.~Pasiliao~Jr.
\newblock An accelerated linearized alternating direction method of
  multipliers.
\newblock {\em SIAM Journal on Imaging Sciences}, 8(1):644--681, 2015.

\bibitem{o2015adaptive}
B.~O’Donoghue and E.~Candes.
\newblock Adaptive restart for accelerated gradient schemes.
\newblock {\em Foundations of computational mathematics}, 15(3):715--732, 2015.

\bibitem{wang2008new}
Y.~Wang, J.~Yang, W.~Yin, and Y.~Zhang.
\newblock A new alternating minimization algorithm for total variation image
  reconstruction.
\newblock {\em SIAM Journal on Imaging Sciences}, 1(3):248--272, 2008.

\bibitem{wibisono2016variational}
A.~Wibisono, A.~C. Wilson, and M.~I. Jordan.
\newblock A variational perspective on accelerated methods in optimization.
\newblock {\em arXiv preprint arXiv:1603.04245}, 2016.

\bibitem{xu2015HHSVM}
Y.~Xu, I.~Akrotirianakis, and A.~Chakraborty.
\newblock Proximal gradient method for huberized support vector machine.
\newblock {\em Pattern Analysis and Applications}, pages 1--17, 2015.

\bibitem{xu2013block}
Y.~Xu and W.~Yin.
\newblock A block coordinate descent method for regularized multiconvex
  optimization with applications to nonnegative tensor factorization and
  completion.
\newblock {\em SIAM Journal on imaging sciences}, 6(3):1758--1789, 2013.

\bibitem{ye2011efficient}
G.-B. Ye, Y.~Chen, and X.~Xie.
\newblock Efficient variable selection in support vector machines via the
  alternating direction method of multipliers.
\newblock In {\em AISTATS}, pages 832--840, 2011.

\end{thebibliography}

\end{document}